\documentclass[12pt]{amsart}
\usepackage[latin1]{inputenc}
\usepackage[english]{babel}
\usepackage{amsmath,amssymb,amsthm,amsfonts, color}
\usepackage{enumerate}
\usepackage{graphicx}
\usepackage{latexsym}
\usepackage[all]{xy}
\usepackage{tikz}

\makeatletter
\@namedef{subjclassname@2010}{ \textup{2010} Mathematics Subject Classification}
\makeatother

\date{}
\newtheorem{lemma}{Lemma}
\newtheorem{remark}{Remark}
\theoremstyle{definition}
\newtheorem{defin}{Definition}
\newtheorem{prop}[lemma]{Proposition}
\newtheorem{cor}[lemma]{Corollary}
\newtheorem{theor}{Theorem}

\begin{document}
\title{ A characterization of $\ell^p$-spaces symmetrically finitely represented in symmetric sequence spaces.}
\thanks{{\rm *}The work was completed as a part of the implementation of the development program of the Scientific and Educational Mathematical Center Volga Federal District, agreement no. 075-02-2020-1488/1.}
\author[Astashkin]{Sergey V. Astashkin}
\address[Sergey V. Astashkin]{Department of Mathematics, Samara National Research University, Moskovskoye shosse 34, 443086, Samara, Russia
}
\email{\texttt{astash56@mail.ru}}
\maketitle

\vspace{-7mm}

\begin{abstract}
For a separable symmetric sequence space $X$ of fundamental type we identify the set ${\mathcal F}(X)$ of all $p\in [1,\infty]$ such that $\ell^p$ is block  finitely represented in the unit vector basis $\{e_k\}_{k=1}^\infty$ of
$X$ in such a way that the unit basis vectors of $\ell^p$ ($c_0$ if $p=\infty$) correspond to pairwise disjoint blocks of $\{e_k\}$ with the same ordered distribution. It turns out that ${\mathcal F}(X)$ coincides with the set of approximate eigenvalues of the operator $(x_k)\mapsto \sum_{k=2}^\infty x_{[k/2]}e_k$ in $X$. In turn, we establish that the latter set is the interval $[2^{\alpha_X},2^{\beta_X}]$, where $\alpha_X$ and $\beta_X$ are the Boyd indices of $X$. As an application, we find the set ${\mathcal F}(X)$ for arbitrary  Lorentz and separable sequence Orlicz spaces.  
\end{abstract}


\newcommand{\fg}{\mathbb N}
\newcommand{\gh}{\mathbb R}

\section{Introduction.}
\label{Intro}

While a Banach space $X$ need not contain a subspace isomorphic to $\ell^p$ for some $1\le p<\infty$ or $c_0$ (as was shown by Tsirel'son \cite{Tsi} in 1974), at least one of these spaces is present in such a space $X$ {\it locally}. More precisely, for every $\varepsilon>0$ there exist finite-dimensional subspaces of $X$ of arbitrarily large dimension $n$ which are $(1+\varepsilon)$-isomorphic to $\ell_p^n$ for some $1\le p<\infty$ or $c_0^n$. This fact is the content of the famous result proved by Krivine in 1976 (we should recall here also the celebrated Dvoretzky theorem, which appeared long before Tsirelson's example, in 1961, showing that for every $\varepsilon>0$ any Banach space contains subspaces $(1+\varepsilon)$-isomorphic to $\ell_2^n$ for each $n\in\fg$; see \cite{Dvor}).
Let us introduce some necessary definitions.

Suppose $X$ is a Banach space, $1\le p\le\infty$, and $\{z_i\}_{i=1}^\infty$ is a bounded sequence in
$X$. The space $\ell^p$ is said to be {\it
block finitely represented in $\{z_i\}_{i=1}^\infty$} if for every
$n\in\mathbb{N}$ and $\varepsilon>0$ there exist $0=m_0<m_1<\dots <m_n$ and 
$\alpha_i\in\mathbb{R}$ such that the vectors $u_k=\sum_{i=m_{k-1}+1}^{m_k}\alpha_iz_i$, $k=1,2,\dots,n$, satisfy the inequality
\begin{equation}
\label{eq: blocks}
(1+\varepsilon)^{-1}\|a\|_p\le \Big\|\sum_{k=1}^n
a_ku_k\Big\|_X\le (1+\varepsilon)\|a\|_p
\end{equation}
for arbitrary $a=(a_k)_{k=1}^n\in\mathbb{R}^n$. In what follows,
$$
\|a\|_p:=\Big(\sum_{k=1}^n |a_k|^p\Big)^{1/p} \enskip\text{if}\enskip 
p<\infty,\enskip\text{and}\enskip
\|a\|_\infty:=\max_{k=1,2,\dots,n}|a_k|
$$
Moreover, the space $\ell^p$, $1\le p\le\infty$, is said to be {\it finitely represented} in a Banach space $X$ if for every $n\in\mathbb{N}$ and $\varepsilon>0$ there exist
$x_1,x_2,\dots,x_n\in X$ such that for any $a=(a_k)_{k=1}^n\in\mathbb{R}^n$
\begin{equation*}
\label{eq: 1}
(1+\varepsilon)^{-1}\|a\|_p\le \Big\|\sum_{k=1}^n
a_kx_k\Big\|_X\le (1+\varepsilon)\|a\|_p.
\end{equation*}

We can state now the above-mentioned famous result proved by Krivine in \cite{Kriv} (see also \cite[Theorem 11.3.9]{AK}).

\begin{theor}\label{Th:Krivine1} 
Let $\{z_i\}_{i=1}^\infty$ be an arbitrary normalized sequence in a Banach space $X$ such that the vectors $z_i$, $i=1,2,\dots$, do not form a relatively compact set. Then $\ell^p$ is block finitely
represented in $\{z_i\}_{i=1}^\infty$ for some $p$, $1\leq p\leq 
\infty$.
\end{theor}

This statement plays a~prominent r\^{o}le in the modern theory of Banach spaces and operators (see, e.g., \cite{JMST}, \cite{MilSch} and \cite{Pisier}). Theorem \ref{Th:Krivine1} is directly connected, in particular,  with the important notions of the Rademacher type and cotype of a Banach space (see, e.g., \cite{AK} and \cite{LT2}). We mention here only the profound result due to Maurey and Pisier (see \cite{MaPi}); it states that, for every infinite-dimensional Banach space $X$, the spaces $\ell_{p_X}$ and $\ell_{q_X}$, where $p_X=\sup\{p\in [1,2]\,:\,X\enskip\text{has type}\enskip p\}$ and
$q_X=\inf\{q\in [2,\infty]\,:\,X\enskip\text{is of cotype}\enskip q\}$,
are finitely represented in $X$. Similar result for Banach lattices and the notions of the upper and lower estimates, when the finite representability of $\ell^p$ in a Banach lattice $X$ has the additional property that the unit basis vectors of $\ell^p$ correspond to pairwise disjoint elements of $X$, was proved by Shepp in \cite{Shepp}. In this connection, it is natural to look for a description of the set of $p$ such that $\ell_p$ is finitely represented in a given Banach space.

In 1978,  Rosenthal \cite{Ros} established the following version of Theorem \ref{Th:Krivine1}.

\begin{theor}\label{Th:Rosenthal} 
Let $\{e_i\}_{i=1}^\infty$ be a subsymmetric unconditional basis for a Banach space $X$ such that $c_0$ is not finitely
represented in $X$. If the numbers $s_n$ are defined by $\|\sum_{i=1}^ne_i\|_X=n^{1/s_n}$, then $\ell^p$ is block finitely represented in $\{e_i\}$ provided that $p$ belongs to the interval $[\liminf_{n\to\infty}s_n,\limsup_{n\to\infty}s_n].$ 
\end{theor}
From an inspection of the proof of Theorem \ref{Th:Rosenthal} (cf. the remark after the proof of Theorem 3.3 in \cite{Ros}) it follows that for every $p$, satisfying the hypothesis, any $n$ and $\varepsilon>0$ there exist pairwise disjoint blocks $u_k$, $k=1,2,\dots,n$, of the basis $\{e_i\}$, with the same ordered distribution, such that we have \eqref{eq: blocks}. We will say that blocks $u=\sum_{i=1}^{m}\alpha_ie_{n_i}$ and $v=\sum_{i=1}^{m}\beta_ie_{k_i}$ of the basis $\{e_i\}$ have {\it the same ordered distribution} if $\beta_i=\alpha_{\pi(i)}$,  $i=1,2,\dots,m$, for some permutation $\pi$ of the set $\{1,2,\dots,m\}$. 

Moreover, Theorem \ref{Th:Rosenthal} suggests that a condition, ensuring that $\ell^p$ is block finitely represented in a subsymmetric unconditional (in particular, symmetric) basis $\{e_i\}$ of a Banach space so that the unit basis vectors of $\ell^p$ correspond to blocks of this basis  with the same ordered distribution, can be in a natural way expressed by using suitable estimates for the norms of an appropriate dilation operator when it is restricted to the set $\{e_i\}$. 

Here, we consider a special class of Banach lattices, the symmetric (in other terminology, rearrangement invariant) sequence spaces $X$ modelled on $\mathbb{N}$ (for all necessary definitions see the next section). We will be interested in finding of {\it all} $p$ such that $\ell^p$ is block finitely represented in the unit vector basic sequence $\{e_i\}_{i=1}^\infty$ in $X$ so that the unit basis vectors of $\ell^p$ ($c_0$ if $p=\infty$) correspond to pairwise disjoint blocks $u_k$, $k=1,2,\dots,n$, of the sequence $\{e_i\}$, with the same ordered distribution. We will say in this case that $\ell^p$ is {\it symmetrically block finitely represented} in the unit vector basic sequence $\{e_i\}_{i=1}^\infty$ in $X$.

We will further adopt also the following terminology. Sequences $x=(x_k)_{k=1}^\infty$ and $y=(y_k)_{k=1}^\infty$ are said to have {\it the same ordered distribution}\footnote{Sometimes such sequences are called {\it equimeasurable}.} if there is a permutation $\pi$ of the set of positive integers such that $y_{\pi(k)}=x_k$ for all $k=1,2,\dots$.


\begin{defin}
Let $X$ be a symmetric sequence space, $1\le p\le\infty$. We say that $\ell^p$ is {\it symmetrically finitely represented} in $X$ if for every $n\in\fg$ and each $\varepsilon>0$ there exist pairwise disjoint sequences $x_k\in X$, $k=1,2,\dots,n$, with the same ordered distribution such that for every $(a_k)_{k=1}^n\in\gh^n$ we have
\begin{equation*}
(1+\varepsilon)^{-1}\|a\|_p\le \Big\|
\sum_{k=1}^na_kx_k\Big\|_X\le (1+\varepsilon) \|a\|_p.
\end{equation*}
\end{defin}
Clearly, if $\ell^p$ is symmetrically block finitely represented in the unit vector basic sequence $\{e_i\}_{i=1}^\infty$ in $X$, then $\ell^p$ is symmetrically finitely represented in $X$.

We will need also the following formally weaker property.

\begin{defin}
We say that $\ell^p$ is {\it crudely symmetrically finitely represented} in a symmetric sequence space $X$ if there exists a constant $C>0$ such that for every $n\in\fg$ we can find pairwise disjoint sequences $x_k\in X$, $k=1,2,\dots,n$, with the same ordered distribution such that for every $(a_k)_{k=1}^n\in\gh^n$ 
\begin{equation*}
C^{-1}\|a\|_p\le \Big\|
\sum_{k=1}^na_kx_k\Big\|_X\le C\|a\|_p.
\end{equation*}
\end{defin}

The set of all $p\in [1,\infty]$ such that  $\ell^p$ is symmetrically finitely represented (resp. crudely symmetrically finitely represented) in $X$ we will denote by ${\mathcal F}(X)$ (resp. ${\mathcal F}_c(X)$).


A key role will be played next by the following notion related to spectral properties of linear operators bounded in Banach spaces.

\begin{defin}
Let  
$T$ be a bounded linear operator on a Banach space $X$. A sequence
$\{u_n\}_{n=1}^\infty\subset
X$, $\|u_n\|=1$, $n=1,2,\dots$, is called an {\it approximate eigenvector} corresponding to an {\it approximate eigenvalue} $\lambda\in\gh$ for $T$ if
$\|Tu_n-\lambda u_n\|_X\to 0$ \footnote{In what follows, we consider   real Banach spaces. Note however that every bounded linear  operator in a Banach spaces over the complex field has at least one approximate eigenvalue (see, e.g., \cite[12.1]{MilSch})}. 
\end{defin}

The first main result of this paper (Theorem \ref{Theorem 1a}) indicates a direct connection of the problem of a description of the set ${\mathcal F}(X)$, for a separable symmetric sequence space $X$, with the identification of the set of approximate eigenvalues of the operator $D:=\tau_1\sigma_2$ in $X$, where $\sigma_2$ and  $\tau_1$ are the dilation and shift operators respectively defined by 
$$
{\sigma}_2 x := \left(x_{[\frac{1+n}{2}]}\right)_{n=1}^{\infty}\;\;\mbox{and}\;\; \tau_1 x:=(x_{n-1})_{n=1}^\infty,\;\;\mbox{where}\; x=(x_{n})_{n=1}^\infty$$
(we set $x_j=0$ if $j\not\in\mathbb{N}$). Specifically, we show that, for every separable symmetric sequence space $X$, $\ell^p$ is symmetrically block finitely represented in the unit vector basis $\{e_k\}$ of $X$ (equivalently, $p\in {\mathcal F}(X)$) if and only if  $2^{1/p}$ is an approximate eigenvalue of the operator $D$ in $X$.

As a consequence of Theorem \ref{Theorem 1a}, we prove that, for any  symmetric sequence space $X$, $\max {\mathcal F}(X)=1/\alpha_X$ and $\min{\mathcal F}(X)=1/\beta_X$ (see Theorem {\rm\ref{Th: Krivine3}). 
Let us mention that this result is stated without a proof in \cite[Theorem 2.b.6]{LT2} for both sequence and function symmetric  spaces (in the case of function symmetric  spaces on $(0,\infty)$ the proof can be found in the paper \cite{A-11}). In particular, from Theorem {\rm\ref{Th: Krivine3} it follows that ${\mathcal F}(X)\ne\emptyset$ for every symmetric sequence space $X$.

Observe that many properties of the dilation operator $\sigma_2$ (and hence those of $D$) in a symmetric space $X$ are determined to a great  extent by the values of the so-called Boyd indices $\alpha_X$ and $\beta_X$ of the space $X$. In particular, it can be rather easily shown (see Corollary \ref{gen case}) that the set of all approximate eigenvalues of the operator $D$  is contained in the interval $[2^{\alpha_X},2^{\beta_X}]$.

The second main result of the paper (Corollary \ref{cor-new}) shows that for a rather wide class of separable symmetric sequence spaces $X$ of fundamental type the set of all approximate eigenvalues of the operator $D$  coincides with the interval $[2^{\alpha_X},2^{\beta_X}]$.  It should be emphasized that the latter condition, imposed on $X$, is not so restrictive; all the  most known and important symmetric sequence spaces, e.g., Orlicz, Lorentz, Marcinkiewicz spaces, are of fundamental type.

As a result, we obtain Theorem \ref{Theorem 4}, which reads that, for any separable symmetric sequence space $X$ of fundamental type, we have ${\mathcal F}(X)={\mathcal F}_c(X)=[1/\beta_X,1/\alpha_X]$. Observe that, in contrast to Theorem \ref{Th:Rosenthal}, we  get rid of the extra condition that $c_0$ is not finitely represented in $X$.

Our approach to the problem of finding the set of approximate eigenvalues of the operator $D$ in a symmetric sequence space $X$ is based on its reduction to the similar task for the shift operator $\tau_1$ in some Banach sequence lattice $E_X$ modelled on $\fg$ such that the dyadic block basis $\{\sum_{i=2^{k-1}}^{2^{k}-1}e_i\}_{k=1}^\infty$ is equivalent in $X$ to the unit vector basis in $E_X$ (see Proposition \ref{Proposition 1}). 
We develop an idea applied by Kalton \cite{Kal} to the special case of Lorentz function spaces and corresponding weighted $\ell^p$-spaces when studying problems related to interpolation theory of operators, in particular, to a characterization of Calder\'{o}n couples of rearrangement invariant spaces (see also \cite[Lemma~2.2]{IK-01}).  

 In the concluding part of the paper, as an application of the results obtained,  we establish that ${\mathcal F}(X)=[1/\beta_X,1/\alpha_X]$ if $X$ is an arbitrary Lorentz space or a separable Orlicz space (see Theorems \ref{Theorem 4a} and \ref{Theorem 4b}).
 
Similar problems for symmetric function spaces on $(0,\infty)$ have been studied in \cite{A-11} and more recently in \cite{A-21}. On the one hand, in this case the situation gets a little more complicated than for sequence spaces; in particular, depending on the behaviour the doubling operator $x(t)\mapsto x(t/2)$ in such a space $X$ of fundamental type, the set of $p$, for which $\ell^p$ is symmetrically finitely represented in $X$, may be also a union of two intervals. On the other hand, from the technical point of view, the case of function spaces is somewhat simpler, because the operator $x(t)\mapsto x(t/2)$ is invertible (in contrast to the operator $D$, which is clearly not surjective).


\vskip0.5cm

\section{Preliminaries.}
\subsection{Banach sequence lattices.}
\label{prel1}
Here, we recall some definitions and results that relate to Banach sequence lattices; for a detailed exposition see, for example, the monographs \cite{KA}, \cite{LT2} and \cite{BSh}.

A Banach space $E$ of sequences of real numbers is said to be a {\it Banach sequence lattice} (or {\it Banach sequence ideal space}) if from $x=(x_k)_{k=1}^\infty\in E$ and $|y_k| \leq |x_k|$, $k=1,2,\dots$, it follows that $y=(y_k)_{k=1}^\infty\in E$ and $\|y\|_E \leq \|x\|_E$.

Note that from the convergence in norm in a Banach sequence lattice $E$ it follows the coordinate-wise convergence (see, e.g., \cite[Theorem~4.3.1]{KA}).

If $E$ is a Banach sequence lattice, then the {\it K\"{o}the dual} 
(or {\it associated}) lattice $E'$ consists of all $y=(y_k)_{k=1}^\infty\in E$ such that
$$
\|y\|_{E'}:=\sup\,\Bigl\{\sum_{k=1}^\infty{x_ky_k}:\;\;
\|x\|_{E}\,\leq{1}\Bigr\}<\infty.
$$
Observe that $E'$ is complete with respect to the norm $y\mapsto \|y\|_{E'}$ and is embedded isometrically into (Banach) dual space $E^*$; moreover, $E'=E^*$ if and only if $E$ is separable \cite[Corollary~6.1.2]{KA}. 

Every Banach sequence lattice $E$ is continuously embedded into its second K\"{o}the dual $E''$ and $\|x\|_{E''}\le\|x\|_E$ for $x\in E$. A Banach sequence lattice $E$ {\it has the Fatou property} (or {\it maximal}) if from $x^{(n)}=(x^{(n)}_k)_{k=1}^\infty\in E,$ $n=1,2,\dots,$ $\sup_{n=1,2,\dots}\|x^{(n)}\|_E<\infty$ and $x^{(n)}_k\to{x_k}$ as $n\to\infty$ for each $k=1,2,\dots$ it follows that $x=(x_k)_{k=1}^\infty\in E$ and $||x||_E\le \liminf_{n\to\infty}{\|x^{(n)}\|_E}.$ A Banach sequence lattice $E$ has the Fatou property if and only if the natural inclusion of $E$ into $E''$ is a surjective isometry \cite[Theorem~6.1.7]{KA}. 


\subsection{Symmetric sequence spaces.}
\label{prel2}
An important class of Banach lattices is formed by the so-called symmetric spaces. 

If $(u_{k})_{k=1}^{\infty }$ is a bounded sequence of real numbers then, in what follows, $(u_{k}^{\ast })_{k=1}^{\infty }$ denotes the nonincreasing permutation of the sequence $(|u_{k}|)_{k=1}^{\infty }$ defined by 
\begin{equation*}
u_{k}^{\ast }:=\inf_{\mathrm{card}\,A=k-1}\sup_{i\in\mathbb{N}\setminus
A}|u_i|,\;\;k\in\mathbb{N}.
\end{equation*}

A Banach sequence lattice $X$ is called \textit{symmetric sequence space} if $X\subset l^\infty$ and
the conditions $y_{k}^{\ast }=x_{k}^{\ast }$, $k=1,2,\dots $, $x=(x_{k})_{k=1}^\infty\in X$ imply that $y=(y_{k})_{k=1}^\infty\in X$ and $\|y\|_{X}=\|x\|_{X}$. In what follows, we will assume that any symmetric space is separable or has the Fatou property.

Clearly, if $X$ is a symmetric sequence space, $a=(a_k)_{k=1}^\infty\in X$ and $\pi$ is an arbitrary permutation of $\mathbb{N}$, then the sequence $a_\pi=(a_{\pi(k)})_{k=1}^\infty$ belongs to $X$ and $\|a_\pi\|_X=\|a\|_X.$

Let $X$ be a symmetric sequence space. The function $\varphi _X(n):=\| \chi _{\{1,2,\dots,n\}}\|_X$, $n\in\fg$, where $\chi_A$ is the characteristic function of a set $A$, is called the {\it fundamental function} of $X$. It is a nondecreasing and positive function such that $\varphi _X(n)/n$ is nonincreasing.

The most important examples of symmetric sequence spaces are the $\ell^p$-spaces, $1\le p\le\infty,$ with the usual norms
$$
{\|a\|}_{\ell^p}:=\left\{
\begin{array}{ll}
{\left( {\sum}_{k=1}^{\infty}{|a_k|}^{p} \right)}^{1/p}\;,\;& 1 \leq p < \infty\\
\sup\limits_{k=1,2,\dots}|a_k|\;,\;& p=\infty
\end{array}. \right.$$
Their generalization, the $\ell^{p,q}$-spaces, $1<p<\infty,$ $1\le q\le\infty,$ are equipped with the quasi-norms
$$
\|a\|_{\ell^{p,q}}:=\Big\{\sum_{k=1}^{\infty}(a_k^*)^qk^{q/p-1}\Big\}^{1/q}\;\;\mbox{if}\;\;1\le q<\infty,$$ 
and
$$
\|a\|_{\ell^{p,\infty}}:=\sup_{k=1,2,\dots} a_k^*k^{1/p}.$$ 
The functional $a\mapsto \|a\|_{\ell^{p,q}}$ does not satisfy the triangle inequality for $p<q\le\infty$, but it is equivalent to a symmetric norm (see, e.g., \cite[Theorem~4.4.3]{BSh}).

In turn, $\ell^{p,q}$-spaces, if $1\le q\le p<\infty$, belong to a wider class of the Lorentz spaces. Let $1\le q<\infty$, and let $\{w_k\}_{k=1}^\infty$ be a nonincreasing sequence of positive numbers.
The Lorentz space $\lambda_q(w)$ (see, e.g., \cite{Lo-51} or \cite[Chapter~4e]{LT1}) consists of all sequences $a=(a_k)_{k=1}^\infty$ satisfying
$$
\|a\|_{\lambda_q(w)}:=\Big(\sum_{k=1}^\infty (a_k^*)^qw_k^q\Big)^{1/q}<\infty.$$
Since the classical Hardy-Littlewood inequality (see, e.g., \cite[Theorem~2.2.2]{BSh}) yields
$$
\|a\|_{\lambda_q(w)}=\sup_{\pi}\Big(\sum_{k=1}^\infty |a_{\pi(k)}|^qw_k^q\Big)^{1/q},$$
where the supremum is taken over all permutations $\pi$ of the set of positive integers, the functional $a\mapsto \|a\|_{\lambda_q(w)}$ defines on the space $\lambda_q(w)$ a symmetric norm.
Every Lorentz sequence space is separable and has the Fatou property.

Clearly, the fundamental function of $\lambda_q(w)$ is defined by
\begin{equation}\label{fund Lor}
\phi_{\lambda_q(w)}(n)=\Big(\sum_{k=1}^n w_k^q\Big)^{1/q},\;\;n\in\fg.\end{equation}
 
Another natural generalization of the $\ell_{p}$-spaces are Orlicz spaces (see \cite[Chapter~4]{LT1}, \cite{KR}, \cite{RR}, \cite{M-89}).
Let $N$ be an Orlicz function, that is, an increasing convex continuous function on $[0,\infty)$ such that $N(0)=0$ and $\lim_{t\to\infty}N(t)=\infty$. The {\it Orlicz sequence space} $l_N$ consists of all sequences $a=(a_k)_{k=1}^\infty$ such that
$$
\|a\|_{l_N}:=\inf\left\{u>0:\,\sum_{k=1}^\infty N\Big(\frac{|a_k|}{u}\Big)\le 1\right\}<\infty.$$
Without loss of generality, we will assume that $N(1) = 1$. In particular, if $N(s)=s^p$, $1\le p<\infty$, we obtain $\ell^p$ with the usual norm. 

Every Orlicz space $l_N$ has the Fatou property and it is separable if and only if the function $N$ satisfies the {\it $\Delta_2$-condition at zero}, i.e.,
$$
\limsup_{u\to 0}\frac{N(2u)}{N(u)}<\infty.$$
The fundamental function of $l_N$ can be calculated by the formula $\phi_{l_N}(n)=1/N^{-1}(1/n)$, $n\in\fg$, where $N^{-1}$ is the inverse function for $N$.

For more detailed information related to symmetric spaces we refer to the books \cite{LT2}, \cite{KPS} and \cite{BSh}.

\smallskip

\subsection{Indices of Banach sequence lattices and symmetric sequence spaces.}
\label{prel3}

Let $E$ be a Banach sequence lattice modelled on $\mathbb{N}$ such that the shift operator $\tau_n(x):=(x_{k-n})_{k=1}^\infty$, where $x=(x_{k})_{k=1}^\infty$, is bounded in $E$ for every $n\in\mathbb{Z}$ (we set $x_j=0$ if $j\not\in\mathbb{N}$). We define the following shift exponents:
$$
k_+(E):=\lim_{n\to\infty}\|\tau_n\|_{E}^{1/n}\;\;\mbox{and}\;\;k_-(E):=\lim_{n\to\infty}\|\tau_{-n}\|_{E}^{1/n}.$$

For each $m \in \mathbb N$ by ${\sigma}_m$ and ${\sigma}_{1/m}$ we define the {\it dilation operators} on the set of all numerical sequences as follows: if $a = (a_n)_{n=1}^\infty$, then
$$
{\sigma}_m a = \big (a_{[\frac{m-1+n}{m}]} \big)_{n=1}^{\infty} 
= \big ( \overbrace {a_1, a_1, \ldots, a_1}^{m}, \overbrace {a_2, a_2, \ldots, a_2}^{m}, \ldots \big)
$$
and
$$
{\sigma}_{1/m} a = \Big (\frac{1}{m} \sum_{k=(n-1)m + 1}^{nm} a_k \Big)_{n=1}^{\infty}.
$$
These operators are bounded in every symmetric sequence space $X$ and $\|\sigma_m\|_{X}\le m$, $\|\sigma_{1/m}\|_{X}\le 1$ for each $m \in \mathbb N$ (see \cite[Theorem~II.4.5]{KPS} or \cite[Chapter 3, Exercise~15, p.~178]{BSh}). This implies, in particular, that $1\le \|D\|_{X}\le 2$, where $D:=\tau_1\sigma_2$, or equivalently $D(x_k)= \sum_{k=2}^\infty x_{[k/2]}e_k$.

The numbers $\alpha_X$ and $\beta_X$ given by 
\begin{equation}\label{C1:Boyd}
\alpha_X:=-\lim\limits_{m\to +\infty}\frac{\log_2 {\|{\sigma}_{1/m}\|_{X}}}{\log_2m}=-\inf\limits_{m\ge 2}\frac{\log_2 {\|{\sigma}_{1/m}}\|_{X}}{\log_2m} 
\end{equation}
and
\begin{equation}\label{C2:Boyd}
\beta_X:=\lim\limits_{m\to +\infty}\frac{\log_2 {\|{\sigma}_{m}\|_{X }}}{\log_2m}=\inf\limits_{m\ge 2}\frac{\log_2 {\|{\sigma}_{m}}\|_{X}}{\log_2m},
\end{equation}
are called the {\it Boyd indices} of $X$. Clearly,
\begin{equation}\label{C3:Boyd}
\alpha_X=-\lim\limits_{n\to +\infty}\frac1n{\log_2 {\|{\sigma}_{2^{-n}}\|}_{X}}\;\;\mbox{and}\;\;\beta_X=\lim\limits_{n\to +\infty}\frac1n{\log_2 {\|{\sigma}_{2^{n}}\|}_{X}}.
\end{equation}

Let $\phi_X=\phi_X(n)$,  $n\in\fg$, be the fundamental function of a symmetric sequence space $X$. Then, we can introduce the  following {\it dilation functions} of $\phi_X$:
$$
M_X^0(n):=\sup_{m\in\fg}\frac{\phi_X(m)}{\phi_X(mn)}\;\;\mbox{and}\;\;M_X^\infty(n):=\sup_{m\in\fg}\frac{\phi_X(mn)}{\phi_X(m)},\;\;n\in\fg,
$$
and the so-called {\it fundamental indices} of  $X$ by
\begin{equation*}
\mu_X= -\lim_{n\to\infty}\frac{\log_2 M_X^0(n)}{\log_2 n}=-\inf_{n\ge 2}\frac{\log_2 M_X^0(n)}{\log_2 n}
\end{equation*}
and
\begin{equation*}
\nu_X= \lim_{n\to\infty}\frac{\log_2 M_X^\infty(n)}{\log_2 n}=\inf_{n\ge 2}\frac{\log_2 M_X^\infty(n)}{\log_2 n}
\end{equation*}
\cite[Exercise~14 for Chapter 3, p.~178]{BSh}. Hence, we can write
\begin{equation*}
\mu_X= -\lim_{n\to\infty}\frac{1}{n}\log_2 M_X^0(2^{n})\;\;\mbox{and}\;\;\nu_X= \lim_{n\to\infty}\frac{1}{n}\log_2 M_X^\infty(2^{n}).
\end{equation*}

From the above definitions it follows easily that $0\leq\alpha_X\le \mu_X\le \nu_X\leq\beta_X\leq 1$ for every symmetric sequence space $X$ (see also \cite[\S\,II.4.4]{KPS}).

We will say that a symmetric sequence space $X$ is of {\it fundamental type} whenever the corresponding Boyd and fundamental indices of $X$ coincide, i.e., $\alpha_X=\mu_{X}$ and $\beta_X=\nu_{X}$.

The most known and important r.i. spaces, in particular, all Lorentz and Orlicz spaces, are of fundamental type \footnote{The first example of a symmetric function space on $[0,1]$ of non-fundamental type was constructed by Shimogaki in \cite{Shimo}.}. 

Let us prove the last assertion for a Lorentz space $\lambda_q(w)$. In other words, taking into account formula \eqref{fund Lor}, we need to show that
\begin{equation}
\label{Boyd ind Lor} 
\alpha_{\lambda_q(w)}=-\lim_{n\to\infty}\frac{1}{n}\log_2\sup_{j\in\fg}\left(\frac{\sum_{k=1}^j w_k^q}{\sum_{k=1}^{2^nj} w_k^q}\right)^{1/q}
\end{equation}
and
\begin{equation}
\label{Boyd ind Lor1} 
\beta_{\lambda_q(w)}= \lim_{n\to\infty}\frac{1}{n}\log_2\sup_{j\in\fg}\left(\frac{\sum_{k=1}^{2^nj} w_k^q}{\sum_{k=1}^{j} w_k^q}\right)^{1/q}.
\end{equation}

We check only \eqref{Boyd ind Lor1}; equality \eqref{Boyd ind Lor} can be proved similarly. Since $(\sigma_{2^n}x)^*= \sigma_{2^n}(x^*)$, we can (and will) assume that a sequence $x=(x_k)_{k=1}^\infty$ is finitely supported and $x=x^*$. Then, applying the Abel transformation, we obtain for some $m\in\fg$ 
$$
\|\sigma_{2^n}x\|_{\lambda_q(w)}^q=\sum_{k=1}^m x_k^q\sum_{i=2^{n}(k-1)+1}^{2^{n}k}w_i^q=\sum_{k=1}^{m-1}(x_k^q-x_{k+1}^q) \sum_{i=1}^{2^{n}k}w_i^q+x_m^q\sum_{i=1}^{2^{n}m}w_i^q.$$ 
Hence, 
\begin{eqnarray*}
\|\sigma_{2^n}x\|_{\lambda_q(w)}^q&\le& \sup_{j\in\fg}\left(\frac{\sum_{i=1}^{2^nj} w_i^q}{\sum_{i=1}^{j} w_i^q}\right)
\left(\sum_{k=1}^{m-1}(x_k^q-x_{k+1}^q) \sum_{i=1}^{k}w_i^q+x_m^q\sum_{i=1}^{{m}}w_i^q\right)\\ &=& \sup_{j\in\fg}\left(\frac{\sum_{i=1}^{2^nj} w_i^q}{\sum_{i=1}^{j} w_i^q}\right)
\|x\|_{\lambda_q(w)}^q,
\end{eqnarray*}
and consequently
$$
\beta_{\lambda_q(w)}\le \lim_{n\to\infty}\frac{1}{n}\log_2\sup_{j\in\fg}\left(\frac{\sum_{i=1}^{2^nj} w_i^q}{\sum_{i=1}^{j} w_i^q}\right)^{1/q}.$$
Since the opposite inequality is immediate, formula \eqref{Boyd ind Lor1} is proved.

The proof of the fact that every Orlicz space is of fundamental type see in \cite{Boyd} or \cite[Theorem~4.2]{Ma-85}. Therefore, since $\phi_{l_N}(n)=1/N^{-1}(1/n)$, $n\in\fg$, where $N^{-1}$ is the inverse function for $N$, the dilation indices of the Orlicz space $l_N$ can be calculated by the formulae:
\begin{equation}
\label{equa24} 
\alpha_{l_N}=-\lim_{n\to\infty}\frac1n\log_2\sup_{k=0,1,\dots}\frac{N^{-1}(2^{-k-n})}{N^{-1}(2^{-k})}\;\;\mbox{and}\;\;\beta_{l_N}=\lim_{n\to\infty}\frac1n\log_2\sup_{k\ge n}\frac{N^{-1}(2^{-k+n})}{N^{-1}(2^{-k})}.
\end{equation}
One can readily check also that
\begin{equation}
\label{equa24om} 
\alpha_{l_N}=\inf\Big\{q:\,\inf_{0<s,t\le 1}\frac{N(st)}{N(s)t^{1/q}}>0\Big\}\;\;\mbox{and}\;\;\beta_{l_N}=\sup\Big\{q:\,\sup_{0<s,t\le 1}\frac{N(st)}{N(s)t^{1/q}}<\infty\Big\}.
\end{equation}


\smallskip

\subsection{Spreading sequence spaces.}
\label{prel4}

A sequence $\{u_n\}_{n=1}^\infty$ in a Banach space $X$ is called {\it spreading} if for every $n\in\fg$ and any integers $0<m_1<m_2<\dots<m_n$ and $a_j\in\mathbb{R}$ we have
$$
\Big\|\sum_{j=1}^na_jx_{m_j}\Big\|_X=\Big\|\sum_{j=1}^na_jx_{j}\Big\|_X.$$

Throughout, we denote by $e_n$, $n\in\mathbb{N}$, the standard unit vectors in sequence spaces and by $c_{0,0}$ the set of all finitely supported sequences, i.e., such that $x_n=0$ for all sufficiently large $n$.

A separable Banach sequence lattice $E$ is {\it spreading} if the unit vector basis $\{e_n\}_{n=1}^\infty$ is spreading. Clearly, every separable symmetric sequence space $X$ is spreading.

For arbitrary $x=(x_n)\in c_{0,0}$, $y=(y_n)\in c_{0,0}$, we denote by $x\oplus y$ their {\it disjoint sum}. This is an arbitrary sequence  whose nonzero entries coincide with all nonzero entries of $x$ and $y$. For instance, if $n_0=\max\{n\in \mathbb{N}\,:\,x_n\ne 0\}$, then for the   disjoint sum of $x$ and $y$ we can take the sequence:
$$
x\oplus y=\sum_{n=1}^{n_0}x_n e_n+\sum_{n=n_0+1}^{\infty}y_{n-n_0}e_n.
$$

Clearly, if $E$ is a spreading Banach sequence lattice, the norm $\|x\oplus y\|_E$ does not depend on a specific choice for $x\oplus y$.

Let $\varepsilon>0$. We say that a sequence $x$ is {\it replaceable} (resp. $\varepsilon$-{\it replaceable}) in a spreading Banach sequence lattice $E$ by a sequence $y$ if for arbitrary $u\in c_{0,0}$, $v\in c_{0,0}$ we have
$$
\|u\oplus x\oplus v\|_E=\|u\oplus y\oplus v\|_E
$$ 
(resp.
$$
\big|\,\|u\oplus x\oplus v\|_E-\|u\oplus y\oplus
v\|_E\,\big|<\varepsilon).
$$

For a more detailed account of the spreading property and related notions see, for instance, \cite[Chapter~11]{AK}.

By $\|T\|_E$ we will denote the norm of an operator $T$ bounded in a Banach space $E$. We write $f\asymp g$ if $cf\leq g\leq Cf$ for some constants $c>0$ and $C>0$ that do not depend on the values of all (or some)  arguments of the functions (quasi-norms) $f$ and $g$. From one appearance to another the value of the constant $C$ may change.

\vskip0.5cm

\section{A connection between symmetric finite representability of $\ell^p$ in symmetric sequence spaces and spectral properties of the operator $D$.}
\label{sec3}

Here, we prove the first main result of the paper, which reduces the problem of identification of the set ${\mathcal F}(X)$, where $X$ is a separable symmetric sequence space, to finding the set of approximate eigenvalues of the operator $D$ in $X$.

\begin{theor}\label{Theorem 1a}
Let $X$ be a separable symmetric sequence space, $1\le p\le\infty$. The following conditions are equivalent:

(a) $\ell^p$ is symmetrically block finitely represented in the unit vector basis $\{e_k\}$ in $X$;

(b) $\ell^p$ is symmetrically finitely represented in $X$;

(c) $\ell^p$ is symmetrically crudely finitely represented in $X$;

(d) $2^{1/p}$ is an approximate eigenvalue of the operator $D:\,X\to X$.
\end{theor}

\begin{proof}

Since implications $(a)\Longrightarrow (b)\Longrightarrow (c)$ are obvious, it remains only to prove that $(c)\Longrightarrow (d)\Longrightarrow (a)$.

$(c)\Longrightarrow (d)$. 
Let $\ell^p$ is symmetrically crudely finitely represented in $X$. Then, by definition, there exists a constant $C>0$ such that for every $n\in\fg$ we can find pairwise disjoint sequences $x_k\in X$, $k=1,2,\dots,2^n$, with the same ordered distribution, satisfying for all $a_k\in \mathbb{R}$ the inequality 
\begin{equation}
\label{eq1} C^{-1}\Big(\sum_{k=1}^{2^n}|a_k|^p\Big)^{1/p}\le \Big\|
\sum_{k=1}^{2^n}a_kx_k\Big\|_X\le C
\Big(\sum_{k=1}^{2^n}|a_k|^p\Big)^{1/p}
\end{equation}
(we assume that $p<\infty$; the case of $p=\infty$ is treated similarly). Since $X$ is a separable symmetric space, we always can ensure by approximation that $x_k\in c_{00}$, $x_k\ge 0$ and, moreover, in view of the definition of the operator $D$ (see Section \ref{prel3}), that the sequences $y_k:=D^{{k-1}}x_1$, $k=1,2,\dots,{n}$, are pairwise disjoint. 
Then, for arbitrary $b_k\in \mathbb{R}$ we have
\begin{equation*}
\Big\|\sum_{k=1}^{n}b_ky_k\Big\|_X=
\Big\|\sum_{k=1}^{n}b_k\sum_{i=2^{k-1}}^{2^k-1}x_i\Big\|_X,
\end{equation*}
and from \eqref{eq1} it follows 
$$
C^{-1} 2^{(k-1)/p}\le\|y_k\|_X\le C\cdot
2^{(k-1)/p}\quad\mbox{and}\quad
C^{-1}n^{1/p}\le\Big\|\sum_{i=1}^n2^{(1-i)/p}y_i\Big\|_X\le
Cn^{1/p}.
$$
Therefore, setting ${z}_k:=2^{(1-k)/p}y_k$, $k=1,2,\dots,n$, and
$$
v_n:=n^{-1/p}\sum_{k=1}^n {z}_k,\quad n=1,2,\dots,
$$
we get
\begin{equation}
\label{eq4} C^{-1}\le\|{z}_k\|_X\le C,\quad
C^{-1}n^{1/p}\le\Big\|\sum_{k=1}^n {z}_k\Big\|_X\le Cn^{1/p}
\end{equation}
and
\begin{equation}
\label{eq5} C^{-1}\le \|v_n\|_X\le C,\quad n=1,2,\dots.
\end{equation}
Moreover, since 
$$
D{z}_k=2^{(1-k)/p}Dy_k=2^{(1-k)/p}y_{k+1}=2^{1/p}{z}_{k+1}, k=1,2,\dots,n-1,$$ 
putting $D_\lambda: =D-\lambda I$, where $\lambda=2^{1/p}$ and $I$  is the identity in $X$, we have
\begin{align*}
D_\lambda v_n &= n^{-1/p}\Big(
\sum_{k=1}^nD{z}_k-\lambda\sum_{k=1}^n{z}_k\Big)
\\
&= n^{-1/p}\Big(\lambda\sum_{k=1}^{n-1}{z}_{k+1}
+ D{z}_n-\lambda\sum_{k=1}^n{z}_k\Big)
\\
&= n^{-1/p}\left(D{z}_n-\lambda{z}_1\right).
\end{align*}
Therefore, taking into account the inequality $1\le\|D\|_{X}\le 2$ (see Section \ref{prel3}) and the first inequality in \eqref{eq4}, we obtain $\|D_\lambda v_n\|_X\le 4Cn^{-1/p}$, whence $\|D_\lambda
v_n\|_X\to 0$ as $n\to\infty$. If now ${w}_n:=v_n/\|v_n\|_X$, by \eqref{eq5}, we see that $\|D_\lambda {w}_n\|_X\to 0$ as well. So, $\lambda=2^{1/p}$ is an approximate eigenvalue of the operator $D$, which completes the proof of (d). 

$(d)\Longrightarrow (a)$.
Suppose $\lambda=2^{1/p}$ is an approximate eigenvalue of the operator $D$ and $\{g_l\}_{l=1}^\infty\subset X$, 
$\|g_l\|_X=1$, $l=1,2,\dots$, is the corresponding approximate eigenvector. 
Since $X$ is separable, there is no loss of generality in assuming that $g_l\in c_{00}$, $l=1,2,\dots$,
and
\begin{equation}
\label{eq8a}
\|Dg_l-\lambda g_l\|_X\le \frac1l,\quad l=1,2,\dots.
\end{equation} 

For every $l\in\fg$, we introduce a new sequence space $\mathcal{E}_l$, which is being the completion of $c_{00}$ with respect to the norm defined by
\begin{equation}
\label{eq9} 
\Big\|\sum_{j=1}^m
a_je_j\Big\|_{\mathcal{E}_l}:=\|a_1g_l\oplus a_2g_l\oplus\dots
\oplus a_mg_l\|_{X},\quad m\in\mathbb{N},
\end{equation}
where $x\oplus y$ is the disjoint sum of sequences $x$ and $y$ (see Section \ref{prel4}). 

Clearly, the unit vector basis $\{e_j\}_{j=1}^\infty$ in $\mathcal{E}_l$ is isometrically equivalent to a sequence of disjoint copies of $g_l$ in $X$. In particular, $\|e_j\|_{\mathcal{E}_l}=\|g_l\|_X=1$ for all $j,l\in\fg$. Hence, arranging arbitrarily all elements of $c_{0,0}$ with rational coordinates in a sequence of collections $(a_1^{(k)},\dots,a_{r_k}^{(k)})_{k=1}^\infty$, we can construct a family of infinite sequences $(l_i^m)_{i=1}^\infty\subset\fg$, $m=1,2,\dots,$ such that $(l_i^1)_{i=1}^\infty\supset(l_i^2)_{i=1}^\infty\supset\dots$ and for every $m=1,2,\dots$ the limit
$$
\lim_{i\to\infty}\Big\|\sum_{j=1}^{r_k} a_j^{(k)}e_j\Big\|_{\mathcal{E}_{{l_i^m}}}
$$
exists for all $1\le k\le m$. The standard diagonal procedure yields then a sequence $(l_s)_{s=1}^\infty$ that is included in each sequence  $(l_i^m)_{i=1}^\infty$ up                                      
to finitely many terms. Hence, routine arguments based on the fact of density of the rationals in $\mathbb{R}$ show that the limit
$$
\lim_{s\to\infty}\Big\|\sum_{j=1}^\infty a_je_j\Big\|_{\mathcal{E}_{{l_s}}}
$$ 
exists for arbitrary $a=(a_j)\in c_{0,0}$. Therefore, we can introduce a new norm $\|a\|_\mathcal{E}$ on $c_{00}$, which is equal to the latter limit. One can easily see that the completion of $c_{00}$ in this norm, denoting by $\mathcal{E}$, is a symmetric sequence space.

Let us prove that for every $\varepsilon>0$ and $m\in\fg$ there exists $s_0\in\fg$ such that for all $s\ge s_0$ and arbitrary $a_j\in\gh$ we have
\begin{equation}
\label{eq7} 
(1+\varepsilon)^{-1}\Big\|\sum_{j=1}^m
a_je_j\Big\|_{\mathcal{E}_{{l_s}}}\le\Big\|\sum_{j=1}^m a_je_j\Big\|_{\mathcal{E}}\le (1+\varepsilon)\Big\|\sum_{j=1}^m a_je_j\Big\|_{\mathcal{E}_{{l_s}}}
\end{equation}
(cf. \cite[Proposition~1.2]{Ros}).

First, for every $s\in\fg$ and all $a_j\in\gh$, $j=1,2,\dots,m$, it holds
\begin{equation}
\label{eq7a} 
\Big\|\sum_{j=1}^m a_je_j\Big\|_{\mathcal{E}}\le m\Big\|\sum_{j=1}^m a_je_j\Big\|_{\mathcal{E}_{{l_s}}}.
\end{equation}
Let $\delta>0$ be such that $\delta(m+2+\delta)<\varepsilon$. Suppose that sequences $(b_j^1), (b_j^2),\dots,(b_j^r)$ form a $\delta$-net of the set 
$$
B:=\Big\{(a_j)_{j=1}^m:\,\sum_{j=1}^m|a_j|\le m\Big\}$$
with respect to the $\ell_1$-norm. Since  
$$
\Big\|\sum_{j=1}^m a_je_j\Big\|_{\mathcal{E}_{{l_s}}}\le\sum_{j=1}^m|a_j|\;\;\mbox{and}\;\;\Big\|\sum_{j=1}^m a_je_j\Big\|_{\mathcal{E}_{{l_s}}}\ge \max_{j=1,2,\dots}|a_j|\ge \frac1m\sum_{j=1}^m|a_j|$$ 
for all $s\in\fg$ and $a_j\in\gh$, $j=1,2,\dots,m$, then 
$$
B\supset \bigcup_{s=1}^\infty \Big\{(a_j)_{j=1}^m:\,\Big\|\sum_{j=1}^m a_je_j\Big\|_{\mathcal{E}_{{l_s}}}=1\Big\}$$
and $\{(b_j^1), (b_j^2),\dots,(b_j^r)\}$ is a $\delta$-net of the surface of the unit ball 
$$
\Big\{(a_j)_{j=1}^m:\,\Big\|\sum_{j=1}^m a_je_j\Big\|_{\mathcal{E}_{{l_s}}}=1\}$$ 
in $\mathcal{E}_{{l_s}}$ with respect to the $\mathcal{E}_{{l_s}}$-norm for every $s\in\fg$. Moreover, by definition of the $\mathcal{E}$-norm, there is $s_0\in\fg$ such that for all $s\ge s_0$ and $k=1,2,\dots,r$ we have 
\begin{equation}
\label{eq7b} 
(1+\delta)^{-1}\Big\|\sum_{j=1}^m
b_j^ke_j\Big\|_{\mathcal{E}_{{l_s}}}\le\Big\|\sum_{j=1}^m b_j^ke_j\Big\|_{\mathcal{E}}\le (1+\delta)\Big\|\sum_{j=1}^m b_j^ke_j\Big\|_{\mathcal{E}_{{l_s}}}.
\end{equation}

Fix $s\ge s_0$. By homogeneity, it suffices to prove inequality \eqref{eq7} in the case when $\|\sum_{j=1}^m a_je_j\|_{\mathcal{E}_{{l_s}}}=1$. First, there is a positive integer $k_0$ such that $1\le k_0\le r$ and 
$$
\Big\|\sum_{j=1}^m (a_j-b_j^{k_0})e_j\Big\|_{\mathcal{E}_{{l_s}}}\le\delta.$$
Then, combining this together with inequalities \eqref{eq7a}, \eqref{eq7b} and the choice of $\delta$, we get
\begin{eqnarray*}
\Big\|\sum_{j=1}^m a_je_j\Big\|_{\mathcal{E}} &\le&\Big\|\sum_{j=1}^m (a_j-b_j^{k_0})e_j\Big\|_{\mathcal{E}}+\Big\|\sum_{j=1}^m b_j^{k_0}e_j\Big\|_{\mathcal{E}}\\
&\le& m \Big\|\sum_{j=1}^m (a_j-b_j^{k_0})e_j\Big\|_{\mathcal{E}_{l_s}}+(1+\delta)\Big\|\sum_{j=1}^m b_j^{k_0}e_j\Big\|_{\mathcal{E}_{{l_s}}}\\&\le& m\delta+(1+\delta)\Big(\Big\|\sum_{j=1}^m a_je_j\Big\|_{\mathcal{E}_{{l_s}}}+\delta\Big)\\ &=& 
m\delta+(1+\delta)^2<1+\varepsilon.
\end{eqnarray*}
Similarly,
\begin{eqnarray*}
\Big\|\sum_{j=1}^m a_je_j\Big\|_{\mathcal{E}} &\ge&\Big\|\sum_{j=1}^m b_j^{k_0}e_j\Big\|_{\mathcal{E}}-\Big\|\sum_{j=1}^m (a_j-b_j^{k_0})e_j\Big\|_{\mathcal{E}}\\
&\ge& (1-\delta)\Big\|\sum_{j=1}^m b_j^{k_0}e_j\Big\|_{\mathcal{E}_{{l_s}}}
-m \Big\|\sum_{j=1}^m (a_j-b_j^{k_0})e_j\Big\|_{\mathcal{E}_{l_s}}\\&\ge& 
(1-\delta)\Big(\Big\|\sum_{j=1}^m a_je_j\Big\|_{\mathcal{E}_{{l_s}}}-\delta\Big)-m\delta\\ &=& 
(1-\delta)^2-m\delta>1-\varepsilon.
\end{eqnarray*}
Thus, inequality \eqref{eq7} is proved.

In view of \eqref{eq9} and \eqref{eq7}, for every $\varepsilon>0$, $m\in\fg$ and any pairwise disjoint sequences $v_k\in c_{00}$, $k=1,2,\dots,m$, with the same ordered distribution there exist pairwise disjoint sequences $w_k\in c_{00}$, $k=1,2,\dots,m$, with the same ordered distribution such that for arbitrary $a_k\in\mathbb{R}$, $k=1,2,\dots,m$,  
\begin{equation}
\label{eq7new} 
(1+\varepsilon)^{-1}\Big\|\sum_{k=1}^m a_kw_k\Big\|_{X}\le
\Big\|\sum_{k=1}^m a_kv_k\Big\|_{\mathcal{E}}\le
(1+\varepsilon)\Big\|\sum_{k=1}^m a_kw_k\Big\|_{X}.
\end{equation}

We show now that the sum $e_1+e_2$ is replaceable in $\mathcal{E}$ by $\lambda e_1$ (see Section~\ref{prel4}). To this end, we first observe that, for all $l\in\fg$, the latter sum is $1/l$-replaceable in $\mathcal{E}_{{l}}$ by $\lambda e_1$. Indeed, if $c=\sum_i c_ie_i\in c_{00}$, $d=\sum_i d_ie_i\in c_{00}$ and $c'=\oplus\sum_i c_ig_{l}$, $d'=\oplus\sum_i d_ig_{l}$, then, taking into account that the space $X$ is symmetric and applying \eqref{eq8a}, we get 
\begin{align*}
&\big|\,\|c\oplus (e_1+e_2)\oplus d\|_{\mathcal{E}_{l}} - \|c\oplus \lambda
e_1\oplus d\|_{\mathcal{E}_{{l}}}\,\big|
\\ 
&\qquad= \big|\,\|c'\oplus
g_{l}\oplus g_{l}\oplus d'\|_{X}
-\|c'\oplus \lambda g_{l}\oplus d'\|_{X}\,\big|
\\
&\qquad= \big|\,\|c'\oplus
Dg_{l}\oplus d'\|_{X}
-\|c'\oplus \lambda g_{l}\oplus d'\|_{X}\,\big|
\\
&\qquad\le\|Dg_{l}-\lambda g_{l}\|_X\le
\frac{1}{l}.
\end{align*}
Thus, by \eqref{eq7}, for all $s\in\fg$ large enough
\begin{align*}
&\big|\|c\oplus (e_1+e_2)\oplus d\|_{\mathcal{E}} - \|c\oplus \lambda
e_1\oplus d\|_{\mathcal{E}}\big|
\\ 
&\quad\le \big|\|c\oplus (e_1+e_2)\oplus d\|_{\mathcal{E}_{l_s}}- \|c\oplus \lambda e_1\oplus
d\|_{\mathcal{E}_{{l_s}}}\big|\\
&\quad + \varepsilon(1+\varepsilon)\big( \|c\oplus (e_1+e_2)\oplus d\|_{\mathcal{E}}+\|c\oplus \lambda e_1\oplus d\|_{\mathcal{E}}\big)
\\
&\quad\le \frac{1}{l_s}+\varepsilon(1+\varepsilon)\big( \|c\oplus (e_1+e_2)\oplus
d\|_{\mathcal{E}}+\|c\oplus \lambda e_1\oplus d\|_{\mathcal{E}}\big).
\end{align*}
Since the right-hand side in the last inequality can be made arbitrarily small, we obtain
$$
\|c\oplus (e_1+e_2)\oplus d\|_{\mathcal{E}} = \|c\oplus \lambda e_1\oplus
d\|_{\mathcal{E}}.
$$

Hence, if $\lambda=1$, it can be easily deduced that $\|e_1+e_2+\dots+e_n\|_{\mathcal{E}}=1$ for all $n\in\fg$, which implies that $\mathcal{E}$ is isometric to $c_0$. Thus, applying \eqref{eq7new} with  $v_k=e_k$, $k=1,2,\dots,n$, for arbitrary $\varepsilon>0$ and $n\in\fg$ we can find pairwise disjoint sequences $w_1,w_2,\dots,w_n$ from $c_{0,0}$ with the same ordered distribution, which satisfy the inequality 
\begin{equation*}
(1+\varepsilon)^{-1}\max_{k=1,2,\dots,n} |a_k|\le
\Big\|\sum_{k=1}^n a_kw_k\Big\|_X\le
(1+\varepsilon)\max_{k=1,2,\dots,n} |a_k|
\end{equation*}
for arbitrary $a_1,a_2,\dots,a_n\in\mathbb{R}$. Since $X$ is a symmetric space, we can assume that $w_k$, $k=1,2,\dots,n$, are subsequent blocks of the unit vector basis, i.e., can be written as $w_k=\sum_{i=m_{k-1}+1}^{m_k}\alpha_iz_i$, where $0=m_0<m_1<\dots <m_n$. Consequently, $c_0$ is symmetrically block finitely represented in the unit vector basis $\{e_k\}$ in $X$, and in this case (a) is proved.
 
Proceed now with the situation when $1<\lambda\le 2$, i.e., $\lambda=2^{1/p}$ with $1\le p<\infty$. Since $\mathcal{E}$ is a symmetric sequence space, the unit vector basis $\{e_n\}_{n=1}^\infty$ is a spreading $1$-unconditional sequence in $\mathcal{E}$. Therefore, thanks to the fact that the sum $e_1+e_2$ is replaceable in $\mathcal{E}$ by $\lambda e_1$, we can apply \cite[Lemma~11.3.11(i)]{AK} to conclude that the $\mathcal{E}$-norm on $\mathcal{E}$ is equivalent to the $\ell^p$-norm. 

Next, following to \cite{Ros} (cf. also \cite[p.~282]{AK}), we will define a variant of the space   $\mathcal{E}$ modelled on the set $\mathcal{Q}_0:=\mathcal{Q}\cap (0,1)$ rather than $\fg$.

Denote by $c_{0,0}(\mathcal{Q}_0)$ the set of all finitely nonzero sequences of real numbers on $\mathcal{Q}_0$. For $u\in c_{0,0}(\mathcal{Q}_0)$ of the form $u=\sum_{k=1}^n a_ke_{q_k}$, where $q_1<q_2<\dots<q_n$, we set
$$
\Big\|\sum_{k=1}^n a_ke_{q_k}\Big\|_{\mathcal{E}(\mathcal{Q}_0)}=\Big\|\sum_{k=1}^n a_ke_{k}\Big\|_{\mathcal{E}}.$$
Let us consider on $\mathcal{E}(\mathcal{Q}_0)$ two linear operators $D_2$ and $D_3$ given by
$$
D_2e_q=e_{q/2}+e_{(q+1)/2},\;\;q\in \mathcal{Q}_0,$$
and
$$
D_3e_q=e_{q/3}+e_{(q+1)/3}+e_{(q+2)/3},\;\;q\in \mathcal{Q}_0.$$
Then, since $\mathcal{E}(\mathcal{Q}_0)$ is a separable space, the operators $D_2$ and $D_3$, being in a sense counterparts of the dilation operators $\sigma_2$ and $\sigma_3$, are bounded on $\mathcal{E}(\mathcal{Q}_0)$ and moreover $1\le \|D_2\|_{\mathcal{E}(\mathcal{Q}_0)}\le 2$ and $1\le \|D_3\|_{\mathcal{E}(\mathcal{Q}_0)}\le 3$. In contrast to $\sigma_2$ and $\sigma_3$, the operators $D_2$ and $D_3$ possess also the following disjointness property.

\begin{lemma}
\label{l1}
For every $l,m\in\fg$ the sequence $D_2^lD_3^m e_{1/6}$ is a sum of $2^l3^m$ different elements of the unit vector basis. Moreover, the sequences $D_2^lD_3^m e_{1/6}$ and $D_2^{l_1}D_3^{m_1} e_{1/6}$ are disjoint if at least one of conditions $l\ne l_1$ or $m\ne m_1$ is satisfied.
\end{lemma}

We postpone for a moment the proof of Lemma \ref{l1} and proceed  with that of the theorem. 


Let
$$
u_n:=n^{-2/p}\sum_{j=1}^n\sum_{k=1}^n 2^{-j/p}3^{-k/p}D_2^jD_3^k e_{1/6},\;\;n=1,2,\dots.$$
Then, according to Lemma \ref{l1},
$$
D_2^jD_3^k e_{1/6}=\sum_{q\in R_{j,k}}e_q,$$
where the sets $R_{j,k}\subset  \mathcal{Q}_0$ are mutually disjoint and ${\rm card}\,R_{j,k}=2^{j}3^{k}$, $j,k=1,\dots,n$. Moreover, since 
$\|D_2^jD_3^ke_{1/6}\|_{\ell^p}^p=2^j3^k$, $j,k=1,\dots,n$, we have $\|u_n\|_{\ell^p}=1$. Observe also that
\begin{eqnarray*}
2^{-1/p}D_2u_n-u_n&=& n^{-2/p}\Big(\sum_{j=1}^n\sum_{k=1}^n 2^{-(j+1)/p}3^{-k/p}D_2^{j+1}D_3^k e_{1/6}-\sum_{j=1}^n\sum_{k=1}^n 2^{-j/p}3^{-k/p}D_2^jD_3^k e_{1/6}\Big)\\&=&
n^{-2/p}\Big(-2^{-1/p}\sum_{k=1}^n 3^{-k/p}D_2D_3^k e_{1/6}+2^{-(n+1)/p}\sum_{k=1}^n 3^{-k/p}D_2^{n+1}D_3^k e_{1/6}\Big). 
\end{eqnarray*}
Hence, taking into account that the sequences $D_2D_3^k e_{1/6}$, $D_2^{n+1}D_3^ke_{1/6}$, $k=1,2,\dots,n$ are pairwise disjoint, one can readily check that
$$
\|2^{-1/p}D_2 u_n-u_n\|_{\ell^p}=2^{1/p}n^{-1/p},\;\;n\in\fg.$$
Similarly, 
$$
\|3^{-1/p}D_3 u_n-u_n\|_{\ell^p}=3^{1/p}n^{-1/p},\;\;n\in\fg.$$
Therefore, since the $\mathcal{E}$- and $\ell^p$-norms are equivalent on $\mathcal{E}$, we get 
$$
\lim_{n\to\infty}\|D_2 u_n-2^{1/p} u_n\|_{\mathcal{E}}=0\;\;\mbox{and}\;\;\lim_{n\to\infty}\|D_3 u_n-3^{1/p} u_n\|_{\mathcal{E}}=0.$$


Next, arguing in the same way as when going from the space $X$ to the space $\mathcal{E}$, we can construct, starting from  $\mathcal{E}$, the space $\mathcal{E}'$ such that the sums $e_1+e_2$ and $e_1+e_2+e_3$ will be replaceable in $\mathcal{E}'$ by the elements
$2^{1/p}e_1$ and $3^{1/p}e_1$, respectively. Then, by \cite[Lemma 11.3.11(ii)]{AK}, the space $\mathcal{E}'$ is {\it isometric} to $\ell^p$  and hence the formulae similar to \eqref{eq9} and \eqref{eq7}, as above, imply that for every $\varepsilon>0$ and $n\in\fg$ there exist pairwise disjoint sequences $u_1,u_2,\dots,u_n$ from $c_{0,0}$ with the same ordered distribution satisfying 
\begin{equation*}
(1+\varepsilon)^{-1}\Big(\sum_{k=1}^n |a_k|^p\Big)^{1/p}\le
\Big\|\sum_{k=1}^n a_ku_k\Big\|_{\mathcal{E}}\le
(1+\varepsilon)\Big(\sum_{k=1}^n |a_k|^p\Big)^{1/p}
\end{equation*}
for arbitrary $a_k\in\mathbb{R}$, $k=1,2,\dots,n$.
Furthermore, taking $u_k$ for $v_k$ in \eqref{eq7new}, for any $\varepsilon>0$ and $n\in\fg$ we can find pairwise disjoint sequences $w_1,w_2,\dots,w_n$ from $c_{0,0}$ with the same ordered distribution such that
$$
(1+\varepsilon)^{-1}\Big\|\sum_{k=1}^n a_ku_k\Big\|_{\mathcal{E}}\le
\Big\|\sum_{k=1}^n a_kw_k\Big\|_{X}\le
(1+\varepsilon)\Big\|\sum_{k=1}^n a_ku_k\Big\|_{\mathcal{E}},$$
where again we can assume that $w_k$, $k=1,2,\dots,n$, are subsequent blocks of the unit vector basis. Combining the last inequalities, we get
$$
(1+\varepsilon)^{-2}\Big(\sum_{k=1}^n |a_k|^p\Big)^{1/p}\le
\Big\|\sum_{k=1}^n a_kw_k\Big\|_{X}\le
(1+\varepsilon)^2\Big(\sum_{k=1}^n |a_k|^p\Big)^{1/p}.$$
Since $\varepsilon>0$ and $n\in\fg$ are arbitrary, this concludes the proof of the theorem.
\end{proof}

\begin{proof}[Proof of Lemma \ref{l1}]
First, for each $q\in\mathbb{Q}_0$ we have
$$
D_2^{l}e_q=\sum_{i=0}^{2^{l}-1}e_{(q+i)2^{-l}}\;\;\mbox{and}\;\;D_3^{m}e_q=\sum_{j=0}^{3^{m}-1}e_{(q+j)3^{-m}}.$$
Therefore, to prove the first assertion of the lemma, it suffices to show that from the equation 
$$
\Big(\frac16+i\Big)2^{-l}+j=\Big(\frac16+i_1\Big)2^{-l}+j_1,$$
for some $l\in\fg$, $m\in\fg$, $1\le i<2^{l}$, $1\le i_1<2^{l}$, $1\le j<3^{m}$ and $1\le j_1<3^{m}$, it follows that $i=i_1$ and $j=j_1$. Indeed, an easy calculation shows that $i-i_1=2^l(j_1-j)$. Since $|i-i_1|<2^{l}$, we deduce $j=j_1$ and hence $i=i_1$. 

Now, we prove the second assertion. Clearly, the supports of the sequences $D_2^lD_3^m e_{q}$ and $D_2^{l_1}D_3^{m_1} e_{q}$ are overlapped if and only if 
$$
\frac{(q+i)2^{-l}+j}{3^m}=\frac{(q+i_1)2^{-l_1}+j_1}{3^{m_1}}$$
or equivalently
$$
\frac{q+i+2^{l}j}{2^{l}3^m}=\frac{q+i_1+2^{l_1}j_1}{2^{l_1}3^{m_1}}$$
for some $l\in\fg$, $m\in\fg$, $l_1\in\fg$, $m_1\in\fg$, $1\le i<2^{l}$, $1\le j<3^{m}$, $1\le i_1<2^{l_1}$ and $1\le j_1<3^{m_1}$ such that at least one of the conditions $l\ne l_1$ or $m\ne m_1$ holds.
Then, inserting into the latter equation $q=1/6$, we obtain
$$
\frac{1+6i+2^{l}\cdot 6j}{2^{l+1}3^{m+1}}=\frac{1+6i_1+2^{l_1}\cdot 6j_1}{2^{l_1+1}3^{m_1+1}}.$$
Observe that both numbers $1+6i+2^{l}\cdot 6j$ and $1+6i_1+2^{l_1}\cdot 6j_1$ are not divisible by $2$ and $3$. Hence, $l=l_1$ and $m=m_1$, and so 
the second assertion of the lemma is proved.
\end{proof}

Next, we are going to prove that the set of approximate eigenvalues of the operator $D$ is contained in the interval $[2^{\alpha_X},2^{\beta_X}]$, where $\alpha_X$ and $\beta_X$ are the Boyd indices of $X$, while $2^{\alpha_X}$ and $2^{\beta_X}$ are always approximate eigenvalues of $D$. 

\begin{prop}\label{prop 2}
Let $X$ be a symmetric sequence space. Then, the operator $D_\lambda=D-\lambda I$, $\lambda>0$, is an isomorphic mapping in $X$ whenever $\lambda\not\in [2^{\alpha_X},2^{\beta_X}]$. Moreover, $D_\lambda$ is not isomorphic if $\lambda$ is equal to $2^{\beta_X}$ or $2^{\alpha_X}$.
\end{prop}
\begin{proof}
First, let $\lambda>2^{\beta_X}$. Then, by the definition of $\beta_X$
(see \eqref{C3:Boyd}), for all sufficiently large positive integers $n$ we have
\begin{equation*}
\label{eq11} \|D^n\|_{X}=\|\sigma_{2^n}\|_{X}<\lambda^{n}.
\end{equation*}
Therefore, the spectral radius $r(D)$ of the operator $D$ satisfies the inequality
$$
r(D)=\lim_{n\to\infty}\|D^n\|_{X}^{1/n}<\lambda,$$
and hence the operator $D_\lambda$ is an isomorphism from $X$ onto $X$.

If $0<\lambda<2^{\alpha_X}$, similarly for all $n\in\mathbb{N}$ large enough it holds 
\begin{equation*}
\label{eq11} \|D^{-n}\|_{X}=\|\sigma_{2^{-n}}\|_{X}\le\lambda^{-n},
\end{equation*}
whence the operator $D^{-1}-\lambda^{-1} I$ is an isomorphism from $X$ onto $X$. Then, in particular, for some $c>0$ and all $x\in X$
$$
\|(D^{-1}-\lambda^{-1} I)x\|_X\ge c\|x\|_X.$$
Consequently, since
$$
D^{-1}D=\sigma_{1/2}\tau_{-1}\tau_1\sigma_{2}=\sigma_{1/2}\sigma_{2}=I,$$
we have
$$
\|D_\lambda x\|_X=\lambda\|(D^{-1}-\lambda^{-1} I)Dx\|_X\ge \lambda c\|Dx\|_X\ge \lambda c\|x\|_X,$$
and the first assertion of the proposition is obtained.

Let us prove the second assertion. Suppose first $\lambda=2^{\beta_X}$. Then, by \eqref{C2:Boyd} (cf. \cite[\S\,II.1.1 and \S\,II.4.3]{KPS}),
\begin{equation*}
\|D^n\|_{X}=\|\sigma_{2^n}\|_{X}\ge 2^{n\beta_X},\quad n\in\fg.
\end{equation*}
Next, as in the proof of Lemma 11.3.12 in \cite{AK}, we write 
$$
\lim_{n\to\infty}\|(n+1)2^{-n\beta_X}D^n\|_{X}=\infty,
$$ 
and then the Uniform Boundedness principle implies the existence of an element $a_0\in X$, $\|a_0\|_X=1$, such that
\begin{equation}
\label{eq12}
\limsup_{n\to\infty}\|(n+1)2^{-n\beta_X}D^na_0\|_{X}=\infty.
\end{equation}
Clearly, $a_0$ may be assumed to be nonnegative.

Since $r(D)=2^{\beta_X}$, the operator $D-rI$ is invertible in 
$X$ if $r>2^{\beta_X}$. Consequently, $(D-rI)^{-2}$ can be represented for such a number $r$ as the following (converging) operator series:
$$
(D-rI)^{-2}=\frac{1}{r^2}\sum_{n=0}^\infty (n+1)r^{-n}D^n.
$$ 
Since $a_0\ge 0$ and $D\ge 0$, we conclude that 
$$
\|(D-rI)^{-2}a_0\|_X\ge r^{-2}\|(n+1)r^{-n}D^na_0\|_{X}
$$ 
for every $r>2^{\beta_X}$ and all $n\in\mathbb{N}$. Hence, by \eqref{eq12},
$$
\lim_{r\to 2^{\beta_X}}\|(D-rI)^{-2}a_0\|_X=\infty.
$$
Clearly, then there exists a sequence $\{r_n\}$, $r_n> 2^{\beta_X}$ and $r_n\to 2^{\beta_X}$ such that either
$$
\lim_{n\to \infty}\|(D-r_nI)^{-1}a_0\|_X=\infty,
$$
or
$$
\lim_{n\to \infty}\frac{\|(D-r_nI)^{-2}a_0\|_X}{
\|(D-r_nI)^{-1}a_0\|_X}=\infty.
$$
It is easy to see that in either case, we can find a sequence $\{g_n\}_{n=1}^\infty\subset X$, $\|g_n\|_X=1$, such
that
$$
\lim_{n\to \infty}\|(D-r_nI)g_n\|_X=0.
$$
Surely, since $r_n\to 2^{\beta_X}$, this implies that the operator $D_\lambda$ is not isomorphic in $X$ if $\lambda=2^{\beta_X}$. 

Let us proceed now with the case when $\lambda=2^{\alpha_X}$. Then, in view of \eqref{C1:Boyd} (cf. \cite[\S\,II.1.1 and \S\,II.4.3]{KPS}), for the operator $D^{-1}=\sigma_{1/2}\tau_{-1}$ it holds
$$
\|(D^{-1})^n\|_{X\to X}=\|\sigma_{2^{-n}}\|_{X\to X}\ge
2^{-n\alpha_X},\quad n\in\fg,
$$ 
and so, as in the preceding case,
\begin{equation}
\label{eq12c}
\limsup_{n\to\infty}\|(n+1)2^{n\alpha_X}D^{-n}b_0\|_{X}=\infty,
\end{equation}
for some $b_0\in X$ such that $\|b_0\|_X=1$ and $b_0\ge 0$.
Since $r(D^{-1})=2^{-\alpha_X}$,  the operator $D^{-1}-t^{-1}I$ is invertible in $X$ if $t<2^{\alpha_X}$. Therefore, one can readily check that for $0<t<2^{\alpha_X}$
\begin{equation}
\label{eq12f}
(D^{-1}-t^{-1}I)^{-2}={t^2}\sum_{n=0}^\infty (n+1)t^{n}D^{-n}.
\end{equation}
Hence,
$$
\|(D^{-1}-t^{-1}I)^{-2}b_0\|_X\ge t^{2}\|(n+1)t^{n}D^{-n}b_0\|_{X}
$$ 
for every $0<t<2^{\alpha_X}$ and all $n\in\mathbb{N}$. Combining this together with \eqref{eq12c}, we get
$$
\lim_{t\to 2^{\alpha_X}}\|(D^{-1}-t^{-1}I)^{-2}b_0\|_X=\infty,
$$
which implies that there exists a sequence $\{t_n\}$, $t_n<2^{\alpha_X}$ and $t_n\to 2^{\alpha_X}$, such that either
\begin{equation}
\label{eq13c}
\lim_{n\to \infty}\|(D^{-1}-t_n^{-1}I)^{-1}b_0\|_X=\infty,
\end{equation}
or
\begin{equation}
\label{eq13d}
\lim_{n\to \infty}\frac{\|(D^{-1}-t_n^{-1}I)^{-2}b_0\|_X}{
\|(D^{-1}-t_n^{-1}I)^{-1}b_0\|_X}=\infty.
\end{equation}

Assume first that we have \eqref{eq13c}. Denoting $\gamma_k:=\|(D^{-1}-t_k^{-1}I)^{-1}b_0\|_X^{-1},$ $k\in\fg$, we define
$$
h_k:=\gamma_k(D^{-1}-t_k^{-1}I)^{-1}b_0.$$
Then, $\gamma_k\to 0$ and hence 
\begin{equation}
\label{eq13e}
\lim_{k\to \infty}\|(D^{-1}-t_k^{-1}I)h_k\|_X=0\;\;\mbox{and}\;\;\|h_k\|_X=1,\;k\in\fg.
\end{equation}
Moreover, since
$$
h_k=\gamma_k(D^{-1}-t_k^{-1}I)^{-1}b_0=\gamma_k\sum_{n=0}^\infty t_k^{n}D^{-n}b_0,
$$
we have
$$
Dh_k=\gamma_k\sum_{n=0}^\infty t_k^{n}D^{-n+1}b_0=\gamma_kDb_0+t_kh_k.
$$
Therefore, it follows that
$$
h_k=t_k^{-1}D(h_k-\gamma_k b_0)=Df_k,\;\;\mbox{where}\;f_k:=t_k^{-1}(h_k-\gamma_k b_0),\;\;k=1,2,\dots,$$ 
and so
$$
(D^{-1}-t_k^{-1}I)h_k=(D^{-1}-t_k^{-1}I)Df_k=-t_k^{-1}(D-t_kI)f_k.
$$
Thus, since $\gamma_k\to 0$ and $t_k\to 2^{\alpha_X}$ as $k\to\infty$, from \eqref{eq13e} it follows that $\|f_k\|_X\asymp 1$ for $k\in\mathbb{N}$ large enough and
$$
\lim_{k\to \infty}\|(D-2^{\alpha_X}I)f_k\|_X=0.
$$
As a result, the operator $D_\lambda$ fails to be isomorphic for $\lambda=2^{\alpha_X}$.

Suppose now that \eqref{eq13c} does not hold. Since
$$
1=\|b_0\|_X\le \|D^{-1}-t_n^{-1}I\|_X\|(D^{-1}-t_n^{-1}I)^{-1}b_0\|_X\le C\|(D^{-1}-t_n^{-1}I)^{-1}b_0\|_X$$
for some $C>0$ and all $n\in\fg$, passing to a subsequence if it is necessary, we can assume that
\begin{equation}
\label{eq13cnot}
0<\liminf_{n\to \infty}\|(D^{-1}-t_n^{-1}I)^{-1}b_0\|_X\le\limsup_{n\to \infty}\|(D^{-1}-t_n^{-1}I)^{-1}b_0\|_X<\infty.
\end{equation}
Moreover, we know that \eqref{eq13d} holds. Consequently, putting
$$
h_k:=\delta_k(D^{-1}-t_k^{-1}I)^{-2}b_0,$$
where $\delta_k:=\|(D^{-1}-t_k^{-1}I)^{-2}b_0\|_X^{-1},$ $k\in\fg$,
as in the preceding case, we get \eqref{eq13e}.
Furthermore, by \eqref{eq12f},
$$
h_k=\delta_k{t_k^2}\sum_{n=0}^\infty (n+1)t_k^{n}D^{-n}b_0,$$
and hence
\begin{eqnarray*}
Dh_k&=&\delta_k{t_k^2}\sum_{n=0}^\infty (n+1)t_k^{n}D^{-n+1}b_0=\delta_k{t_k^2}\Big(Db_0+t_k\sum_{j=0}^\infty (j+2)t_k^{j}D^{-j}b_0\Big)\\
&=&t_kh_k+\delta_k{t_k^2}\Big(Db_0+t_k\sum_{j=0}^\infty t_k^{j}D^{-j}b_0\Big)=t_kh_k+\delta_k{t_k^2}Db_0+t_kv_k,
\end{eqnarray*}
where
$$
v_k:=\delta_k{t_k^2}\sum_{j=0}^\infty t_k^{j}D^{-j}b_0=\delta_k{t_k^2}(D^{-1}-t_k^{-1}I)^{-1}b_0.$$
Therefore,
$$
h_k=t_k^{-1}Df_k-v_k,\;\;\mbox{where}\;\;f_k:=h_k-\delta_k{t_k^2}b_0,\;\;k=1,2,\dots,$$
whence we get
\begin{eqnarray}
(D^{-1}-t_k^{-1}I)h_k&=&t_k^{-1}(D^{-1}-t_k^{-1}I)Df_k-(D^{-1}-t_k^{-1}I)v_k\nonumber\\
&=&-t_k^{-2}(D-t_kI)f_k-(D^{-1}-t_k^{-1}I)v_k.
\label{eq13g}
\end{eqnarray}
Combining \eqref{eq13cnot} and \eqref{eq13d}, we infer that $\delta_k\to 0$, which implies that $\|f_k\|_X\asymp 1$ for all $k\in\mathbb{N}$ large enough and
$$
\lim_{k\to\infty}\|(D^{-1}-t_k^{-1}I)v_k\|_X=\lim_{k\to\infty}\delta_k{t_k^2}=0.$$ 
Hence, from \eqref{eq13g}, \eqref{eq13e} and the fact that $t_k\to 2^{\alpha_X}$ it follows that
$$
\lim_{k\to\infty}\|(D-2^{\alpha_X}I)f_k\|_X=0,$$
whence $D_\lambda$ is not isomorphic for $\lambda=2^{\alpha_X}$. This completes the proof.
\end{proof}

From Proposition \ref{prop 2} it follows

\begin{cor}
\label{gen case}
Suppose $X$ is a symmetric sequence space. Then, the set of approximate eigenvalues of the operator $D$ is contained in the interval $[2^{\alpha_X},2^{\beta_X}]$ and the numbers $2^{\alpha_X}$ and  $2^{\beta_X}$  are approximate eigenvalues of $D$. 
\end{cor} 

We get also the following result, which implies in particular that the set ${\mathcal F}(X)$ is non-empty for every symmetric sequence space $X$.

\begin{theor}
\label{Th: Krivine3}
If $X$ is an arbitrary symmetric sequence space, then $\max {\mathcal 
F}(X)=1/\alpha_X$ and $\min{\mathcal F}(X)=1/\beta_X$.
\end{theor}

\begin{proof} 
If $X$ is separable, the claim follows from Theorem \ref{Theorem 1a} and Proposition \ref{prop 2}.
Otherwise, $X$ has the Fatou property. One can easily see that then 
$\|\sigma_m\|_{X}=\|\sigma_m\|_{X_0}$ and $\|\sigma_{1/m}\|_{X}=\|\sigma_{1/m}\|_{X_0}$ for each $m\in\fg$, where $X_0$ is the separable part of $X$ (see Section \ref{prel2}). Thus, $\alpha_X=\alpha_{X_0}$ and $\beta_X= \beta_{X_0}$. We can assume that $E\ne L_\infty$. Therefore, $X_0$ is a separable symmetric space and so $\max {\mathcal F}(X_0)=1/\alpha_X$ and $\min{\mathcal F}(X_0)=1/\beta_X$. Since $X_0$ is a subspace of $X$, then we infer that $\max {\mathcal F}(X)\ge 1/\alpha_X$ and $\min{\mathcal F}(X)\le 1/\beta_X$. On the other hand, if $\ell^p$ is symmetrically finitely represented in   $X$, then from the definition of the Boyd indices it follows immediately that $1/\beta_X\le p\le 1/\alpha_X$, and hence the desired result follows.
\end{proof}

\vskip0.5cm

\section{On some correspondence between symmetric sequence spaces and Banach sequence lattices.}
\label{dyadic block basis}

As follows from Theorem \ref{Theorem 1a}, to find the set of all $p$ such that $\ell^p$ is symmetrically block finitely represented in the unit vector basis of a separable symmetric sequence space, it suffices to identify the set of all approximate eigenvalues of the dilation operator $D$. In this section, we show that the latter problem reduces to a similar task for the shift operator in a certain Banach sequence lattice.

We will associate to a symmetric sequence space $X$ the Banach sequence lattice $E_X$ equipped with the norm
$$
\|a\|_{E_X}:=\Big\|\sum_{k=1}^\infty a_{k}\sum_{i=2^{k-1}}^{2^{k}-1}e_i\Big\|_X.$$

We begin with a lemma which establishes simple connections between the Boyd indices $\alpha_X$ and $\beta_X$ of a symmetric sequence space $X$ and the shift exponents $k_+(E_X)$ and $k_-(E_X)$ of the lattice $E_X$.

\begin{lemma}
\label{dilation}
For every symmetric sequence space $X$ we have 
\begin{equation}\label{equa: 1001}
\|\tau_n\|_{E_X}\le 2\|\sigma_{2^n}\|_X,\;\;n\in\mathbb{Z},
\end{equation}
and
\begin{equation}\label{equa: 1001a}
\|\sigma_{2^n}\|_X\le\|\tau_{-n+1}\|_{E_X},\;\;n\in\mathbb{Z}.
\end{equation}
Therefore, $k_-(E_X)=2^{-\alpha_X}$ and $k_+(E_X)=2^{\beta_X}$.
\end{lemma}
\begin{proof}
Let $n\in\mathbb{N}$ and let $a=(a_{k})_{k=1}^\infty\in E_X$. It may be assumed that $a_k\ge 0$, $k=1,2,\dots$ Then, 
\begin{eqnarray*}
\|\tau_na\|_{E_X}&=&\Big\|\sum_{k=n+1}^\infty a_{k-n}\sum_{i=2^{k-1}}^{2^{k}-1}e_i\Big\|_X=\Big\|\sum_{k=1}^\infty a_{k}\sum_{i=2^{n+k-1}}^{2^{n+k}-1}e_i\Big\|_X\\&=&\Big\|\sigma_{2^n}\Big(\sum_{k=1}^\infty a_{k}\sum_{i=2^{k-1}}^{2^{k}-1}e_i\Big)\Big\|_X\le \|\sigma_{2^n}\|_X\|a\|_{E_X}.
\end{eqnarray*}

Before proving the same estimate for negative integers, recall that for arbitrary $n\in\mathbb{N}$
$$
\sigma_{2^{-n}}x:=2^{-n}\sum_{k=1}^\infty\sum_{i=1}^{2^n} x_{(k-1)2^n+i}e_k,\;\;\mbox{where}\;\;x=(x_k)_{k=1}^\infty$$
(see Section \ref{prel3}). Thus, we have
$$
\tau_{-n}a=\sum_{k=1}^\infty a_{k+n}\sum_{i=2^{k-1}}^{2^{k}-1}e_i,$$
while
\begin{eqnarray*}
\sigma_{2^{-n}}\Big(\sum_{k=1}^\infty a_{k}\sum_{i=2^{k-1}}^{2^{k}-1}e_i\Big)&\ge&\sum_{k=1}^\infty 2^{-n}\left((2^n-1)a_{k+n}\right)\sum_{i=2^{k-1}+1}^{2^{k}}e_i\\&\ge&\frac12\sum_{k=1}^\infty a_{k+n}\sum_{i=2^{k-1}+1}^{2^{k}}e_i.
\end{eqnarray*}
Therefore, since $X$ is a symmetric space, we get  
$$
\|\tau_{-n}a\|_{E_X}\le 2\Big\|\sigma_{2^{-n}}\Big(\sum_{k=1}^\infty a_{k}\sum_{i=2^{k-1}}^{2^{k}-1}e_i\Big)\Big\|_X\le 2\|\sigma_{2^{-n}}\|_X\|a\|_{E_X},$$
and inequality \eqref{equa: 1001} is proved.

For the converse direction, we will need the natural averaging projection $Q$ of $X$ onto $E_X$ defined by
\begin{equation}\label{projection Q}
Qx:=\sum_{k=1}^\infty 2^{-k+1}\sum_{j=2^{k-1}}^{2^{k}-1}x_j\sum_{i=2^{k-1}}^{2^{k}-1}e_i,\;\;\mbox{where}\;\;x=(x_{k})_{k=1}^\infty.
\end{equation}
It is well known that $Q$ has norm $1$ in each symmetric space $X$ (see, e.g., \cite[\S\,II.3.2]{KPS}). 

Let $n\in\mathbb{N}$ and $a=(a_{j})_{j=1}^\infty\in X$ be a  nonincreasing and nonnegative sequence.
Then, we have
\begin{eqnarray}
\|\tau_{n+1}Qa\|_{E_X}&=&\Big\|\sum_{l=n+2}^\infty 2^{n-l+2}\sum_{j=2^{l-n-2}}^{2^{l-n-1}-1}a_je_l\Big\|_{E_X}\nonumber\\&=&
\Big\|\sum_{l=n+2}^\infty 2^{n-l+2}\sum_{j=2^{l-n-2}}^{2^{l-n-1}-1}a_j\sum_{i=2^{l-1}}^{2^{l}-1}e_i\Big\|_{X}\nonumber\\
&=&\Big\|\sum_{l=1}^\infty 2^{-l+1}\sum_{j=2^{l-1}}^{2^{l}-1}a_j\sum_{i=2^{l+n}}^{2^{l+n+1}-1}e_i\Big\|_X.
\label{eq: rearrang-new}
\end{eqnarray}
Show that
\begin{equation}\label{equa: rearrang}
\sum_{k=1}^\infty a_k\sum_{i=2^{n}(k-1)+1}^{2^{n}k}e_i\le \sum_{l=1}^\infty 2^{-l+1}\sum_{j=2^{l-1}}^{2^{l}-1}a_j\sum_{i=2^{l+n}}^{2^{l+n+1}-1}e_i.
\end{equation}
Since both sequences are nonincreasing, it suffices to check that for each $i\in\fg$ the sequence from the right-hand side of \eqref{equa: rearrang}  contains at least $2^ni$ entries that are larger than or equal to $a_i$. 

Let $i\in\fg$ be arbitrary. Since $a=(a_{j})_{j=1}^\infty\in X$ is   nonincreasing, the inequality $2^l-1\le i$ ensures that $a_i\le 2^{-l+1}\sum_{j=2^{l-1}}^{2^{l}-1}a_j$.  If $l_0=l_0(i)$ is the most positive integer  satisfying the latter estimate, then $2^{l_0+1}> i+1$. Therefore, the number of entries in the sequence from the right-hand side of \eqref{equa: rearrang}, which are not less than $a_i$,  satisfies the estimate
$$
\sum_{l=1}^{l_0}2^{l+n}=2^{l_0+n+1}-1\ge (i+1)2^{n}-1\ge 2^ni.$$
As was said above, this implies \eqref{equa: rearrang}.

Combining \eqref{eq: rearrang-new} and \eqref{equa: rearrang} together with the fact that the left-hand side of \eqref{equa: rearrang} coincides with the sequence $\sigma_{2^n}a$,
we infer 
$$
\|\sigma_{2^n}a\|_X\le \|\tau_{n+1}Qa\|_{E_X}\le \|\tau_{n+1}\|_{E_X}\|Qa\|_{E_X}=\|\tau_{n+1}\|_{E_X}\|a\|_{X}.$$
Thus, taking into account that $(\sigma_{2^n}a)^*=\sigma_{2^n}(a^*)$, $n\in\mathbb{N}$, we get \eqref{equa: 1001a} for positive $n$.

Next, observe that for every $n\in\fg$
\begin{equation}
\label{equa102}
\sigma_{2^{-n}}a=\sum_{k=1}^\infty c_ke_k,\;\;\mbox{where}\;\;c_k:=2^{-n}\sum_{i=2^{n}(k-1)+1}^{2^{n}k}a_i,\;\;k=1,2,\dots, 
\end{equation}
and hence $\sigma_{2^{-n}}a=\sigma_{2^{-n}}R_na$, where $R_n$ is the norm one averaging projection on $X$ defined by
\begin{equation*}
R_nx:=\sum_{k=1}^\infty 2^{-n}\Big(\sum_{i=2^{n}(k-1)+1}^{2^{n}k}x_i\Big) \sum_{i=2^{n}(k-1)+1}^{2^{n}k}e_i,\;\;\mbox{where}\;\;x=(x_{k})_{k=1}^\infty
\end{equation*}
(see, e.g., \cite[\S\,II.3.2]{KPS}). 
On the other hand,
\begin{equation}
\label{equa102a}
\|\tau_{-n+1}QR_na\|_{E_X}=\Big\|\sum_{k=1}^\infty (QR_na)_{k+n-1}\sum_{1=2^{k-1}}^{2^{k}-1}e_i\Big\|_{X}=\Big\|\sum_{j=n}^\infty (QR_na)_{j}\sum_{1=2^{j-n}}^{2^{j-n+1}-1}e_i\Big\|_{X}.
\end{equation}

Assuming that $a=a^*$, we prove that
\begin{equation}
\label{equa102b}
\sum_{k=1}^\infty c_ke_k\le\sum_{j=n}^\infty (QR_na)_{j}\sum_{1=2^{j-n}}^{2^{j-n+1}-1}e_i.
\end{equation}
To this end, again it suffices to check that for each $k\in\fg$ the sequence from the right-hand side contains at least $k$ entries that are larger than or equal to $c_k$. 

Let $k\in\mathbb{N}$ be fixed. From the definition of the operators $Q$ and $R_n$ it follows that $(QR_na)_{j}\ge c_k$ whenever $2^j-1\le k2^n$. Now, if $j_0=j_0(k)$ is the most positive integer satisfying this inequality, then  $2^{j_0+1}>k2^n$ or $2^{j_0-n+1}>k$. 
Thus, the number of entries in the sequence from the right-hand side of \eqref{equa102b},
which are not less than $c_k$, is more than or equal to
$$
\sum_{j=n}^{j_0}2^{j-n}=2^{1+j_0-n}> k,$$
and so we get \eqref{equa102b}.

As a result, from \eqref{equa102} --- \eqref{equa102b} it follows that
\begin{eqnarray*}
\|\sigma_{2^{-n}}a\|_X &\le& \|\tau_{-n+1}QR_na\|_{E_X}\le \|\tau_{-n+1}\|_{E_X}\|QR_na\|_{E_X}\\ &=&\|\tau_{-n+1}\|_{E_X}\|QR_na\|_{X}\le \|\tau_{-n+1}\|_{E_X}\|a\|_{X}.
\end{eqnarray*}
Since $(\sigma_{2^{-n}}a)^*\le\sigma_{2^{-n}}(a^*)$, $n\in\mathbb{N}$, and $X$ is a symmetric space, the latter inequality holds for every $a\in X$. Hence, as a result, we obtain inequality \eqref{equa: 1001a} for negative $n$.

To complete the proof, it remains to observe that coincidence of the   Boyd indices of $X$ with the corresponding shift exponents of $E_X$  follows immediately from their definition and inequalities \eqref{equa: 1001} and \eqref{equa: 1001a}.
\end{proof}

Sometimes (see Section \ref{Lor and Orl} below) it is useful to know that a lattice of the form $E_X$ can be obtained also in a different way, when one starts from a Banach sequence lattice. 

Let us note first that for every symmetric sequence space $X$ and all $x\in X$ it holds
\begin{equation}
\label{second view} 
\|x\|_{X}\le \Big\|\sum_{k=0}^\infty x^*_{2^k}e_{k+1}\Big\|_{E_X}\le 5\|x\|_{X}.
\end{equation}
Indeed, assuming (as we can) that $x=x^*$, by the definition of $E_X$,
we get
$$
\Big\|\sum_{k=0}^\infty x_{2^k}e_{k+1}\Big\|_{E_X}=\Big\|\sum_{k=0}^\infty x_{2^k}\sum_{i=2^{k}}^{2^{k+1}-1}e_i\Big\|_{X}\ge \Big\|\sum_{i=1}^\infty x_{i}e_i\Big\|_{X}=\|x\|_X,$$
which gives the left-hand inequality. 

Conversely, from \eqref{equa: 1001} and the estimate $\|\sigma_2\|_{X\to X}\le 2$ \cite[Theorem~II.4.5]{KPS} it follows that
\begin{eqnarray*}
\Big\|\sum_{k=0}^\infty x_{2^k}e_{k+1}\Big\|_{E_X}&\le& \|x_1e_1\|_{E_X}+ \Big\|\sum_{k=1}^\infty x_{2^k}e_{k+1}\Big\|_{E_X}\le
\|x\|_X+\Big\|\tau_1\Big(\sum_{k=1}^\infty x_{2^k}e_{k}\Big)\Big\|_{E_X}\\
&\le& \|x\|_X+\|\tau_1\|_{E_X}\Big\|\sum_{k=1}^\infty x_{2^k}\sum_{i=2^{k-1}}^{2^{k}-1}e_i\Big\|_{X}\\&\le& \|x\|_X+2\|\sigma_2\|_{X\to X}\|x\|_X\le 5 \|x\|_X,
\end{eqnarray*}
and the right-hand side inequality in \eqref{second view} follows as well.

In the next result we show that, under some conditions, inequalities similar to \eqref{second view} for a Banach lattice $E$ imply that $E=E_X$ with equivalence of norms. Although this result is in fact known (cf. \cite[Proposition~5.1]{Kal}), taking into account that in the case of sequence lattices it is stated in \cite{Kal} without   proofs, we provide here its complete proof for convenience of the reader.

\begin{prop}
\label{prop:latttice and symm}
Let $E$ be a Banach sequence lattice, with $k_-(E)<1$,  and let $X$ be a symmetric sequence space satisfying
\begin{equation}
\label{equa10} 
\|a\|_{X}\asymp \Big\|\sum_{k=0}^\infty a^*_{2^k}e_{k+1}\Big\|_E.
\end{equation}
Then, $E_X=E$ (with equivalence of norms).
\end{prop}
\begin{proof}

At first, observe that the norms $\|a\|_{E_X}$ and $\|a\|_{E}$ are equivalent if $a=(a_{k})_{k=1}^\infty$ is a decreasing nonnegative sequence (with the equivalence constant from \eqref{equa10}). Indeed, since $(\sum_{k=1}^\infty a_k\sum_{i=2^{k-1}}^{2^{k}-1}e_i)^*_{2^j}=a_{j+1}$, $j=0,1,\dots$, by the definition of $E_X$ and \eqref{equa10}, we have
$$
\|a\|_{E_X}=\Big\|\sum_{k=1}^\infty a_{k}\sum_{i=2^{k-1}}^{2^{k}-1}e_i\Big\|_X\asymp\Big\|\sum_{j=0}^\infty a_{j+1}e_{j+1}\Big\|_E=\|a\|_{E}.$$

Assume now that $a=(a_{k})_{k=1}^\infty\in E_X$ is arbitrary. The condition
$k_-(E)<1$ ensures that for some $\gamma>0$ and $C>0$  
\begin{equation}
\label{equa10a} 
\|\tau_{-j}\|_{E\to E}\le C2^{-\gamma j},\;\;j\in\mathbb{N}.
\end{equation}
Consequently, since the sequence $\left(\max_{j\ge 0}|\tau_{-j}a|)_k\right)$, $k=1,2,\dots$, is decreasing and
$$
|a_k|\le (\max_{j\ge 0}|\tau_{-j}a|)_k,\;\;k=1,2,\dots,$$
by the above observation, we have
\begin{eqnarray*}
\|a\|_{E_X} &\le& \left\|\left(\max_{j\ge 0}|\tau_{-j}a|\right)_k\right\|_{E_X}\asymp 
\left\|\left(\max_{j\ge 0}|(\tau_{-j}a|\right)_k\right\|_{E}\\ &\le& \Big\|\Big(\sum_{j=0}^\infty |\tau_{-j}a|\Big)_k\Big\|_{E}\le 
\sum_{j=0}^\infty \|\tau_{-j}\|_{E\to E} \|a\|_E= C\|a\|_E. 
\end{eqnarray*}

Conversely, for each $a=(a_{k})_{k=1}^\infty\in E$ we have
\begin{eqnarray}
\|a\|_{E} &\le& \left\|\left(\max_{j\ge 0}|\tau_{-j}a|\right)_k\right\|_{E}\asymp 
\left\|\left(\max_{j\ge 0}|\tau_{-j}a|\right)_k\right\|_{E_X}\nonumber\\ &=& \Big\|\sum_{k=1}^\infty \left(\max_{j\ge 0}|\tau_{-j}a|\right)_k\sum_{i=2^{k-1}}^{2^{k}-1}e_i\Big\|_{X}\nonumber\\&\le&
\Big\|\sum_{k=1}^\infty\sum_{j=0}^\infty (|\tau_{-j}a|)_k\sum_{i=2^{k-1}}^{2^{k}-1}e_i\Big\|_{X}\nonumber\\&\le&
\sum_{j=0}^\infty\Big\|\sum_{k=1}^\infty (|\tau_{-j}a|)_k\sum_{i=2^{k-1}}^{2^{k}-1}e_i\Big\|_{X}. 
\label{equa11}
\end{eqnarray}

Next, by the definition of the operator $\sigma_{2^{-n}}$, we have
$$
\sum_{k=1}^\infty (|\tau_{-j}a|)_k\sum_{i=2^{k-1}}^{2^{k}-1}e_i =
\sum_{k=1}^\infty|a_{k+j}|\sum_{i=2^{k-1}}^{2^{k}-1}e_i=
\sum_{r=j+1}^\infty|a_{r}|\sum_{i=2^{r-j-1}}^{2^{r-j}-1}e_i=\sigma_{2^{-j}}a^{(j)},
$$
where $a^{(j)}=\sum_{k=j+1}^\infty |a_{k}|\sum_{i=2^{k-1}}^{2^{k}-1}e_{i-2^j+1}$.
Therefore, since for every $j=0,1,\dots$
$$
(a^{(j)})^*\le \Big(\sum_{k=1}^\infty a_{k}\sum_{i=2^{k-1}}^{2^{k}-1}e_{i}\Big)^*$$
and $(\sigma_{2^{-n}}b)^*\le \sigma_{2^{-n}}(b^*)$, $n\in\mathbb{N}$,
 from \eqref{equa11} it follows that
\begin{equation}
\label{equa12} 
\|a\|_{E} \le \sum_{j=0}^\infty\Big\|\sigma_{2^{-j}}\Big(\sum_{k=1}^\infty a_{k}\sum_{i=2^{k-1}}^{2^{k}-1}e_i\Big)^*\Big\|_{X}.
\end{equation}

We claim that there exists a constant $C > 0$ such that for every nonincreasing and nonnegative sequence $b=(b_{k})_{k=1}^\infty\in X$ and all integers $j\ge 0$
\begin{equation}
\label{equa13} 
\|\sigma_{2^{-j}}b\|_X\le C\|\tau_{-j+1}\|_{E\to E}\|b\|_X.
\end{equation}

Indeed, by \eqref{equa10} and the inequality $(2^k-1)2^j+i\ge 2^{k+j-1}$ if $k,i\ge 1$, we have
\begin{eqnarray*}
\|\sigma_{2^{-j}}b\|_X &\asymp& \Big\|2^{-j}\sum_{k=0}^\infty\sum_{i=1}^{2^j} b_{(2^k-1)2^j+i}e_{k+1}\Big\|_E\\ &\le&
\Big\|b_1e_1+\sum_{k=1}^\infty b_{2^{k+j-1}}e_{k+1}\Big\|_E\\ &=&
\Big\|b_1e_1+\tau_{{-j+1}}\Big(\sum_{k=1}^\infty b_{2^{k}}e_{k+1}\Big)\Big\|_E\\ &\le&
\|b\|_X+\|\tau_{{-j+1}}\|_{E\to E}\Big\|\sum_{k=1}^\infty b_{2^{k}}e_{k+1}\Big\|_E\\
&\le& C\|\tau_{{-j+1}}\|_{E\to E}\|b\|_X,
\end{eqnarray*}
and \eqref{equa13} is proved. 

Applying estimate \eqref{equa13} to the sequence $b=\sum_{k=1}^\infty a_{k}\sum_{i=2^{k-1}}^{2^{k}-1}e_i$, we get
\begin{eqnarray*}
\Big\|\sigma_{2^{-j}}\Big(\sum_{k=1}^\infty a_{k}\sum_{i=2^{k-1}}^{2^{k}-1}e_i\Big)^*\Big\|_X &\le& C\|\tau_{-j+1}\|_{E\to E}\Big\|\sum_{k=1}^\infty a_{k}\sum_{i=2^{k-1}}^{2^{k}-1}e_i\Big\|_X\nonumber\\
&=& C\|\tau_{-j+1}\|_{E\to E}\|a\|_{E_X},
\label{equa14}
\end{eqnarray*}
and so from \eqref{equa12} and \eqref{equa10a} it follows that
$$
\|a\|_E \le C\sum_{j=0}^\infty\|\tau_{-j+1}\|_{E\to E}\|a\|_{E_X}= C'\|a\|_{E_X}.
$$
This completes the proof.
\end{proof}

From Proposition \ref{prop:latttice and symm} and Lemma  \ref{dilation}
it follows

\begin{cor}
\label{cor:basic}
Let $X$ be a symmetric sequence space such that equivalence \eqref{equa10} holds for some Banach sequence lattice $E$ with $k_-(E)<1$. Then, $\alpha_X>0$.
\end{cor}


\smallskip

In the concluding part of this section, we establish a direct connection between spectral properties of the dilation operator $D:=\tau_1\sigma_2$ in a symmetric sequence space $X$ and the shift operator $\tau_1$ in $E_X$. 

Let ${{D}_\lambda}:={{D}}-\lambda I$ and ${T_\lambda}:={\tau_1}-\lambda I$, $\lambda>0$. 
Setting
$$
Sx:=\sum_{k=1}^\infty x_k\sum_{i=2^{k-1}}^{2^k-1}e_i,\;\;\mbox{where}\;x=(x_k)_{k=1}^\infty,$$ 
we get
\begin{eqnarray}
{{D}_\lambda} {Sa} &=&\sum_{k=1}^\infty a_k \sum_{i=2^{k}}^{2^{k+1}-1} e_i - \lambda \sum_{k=1}^\infty a_k\sum_{i=2^{k-1}}^{2^k-1}e_i\nonumber\\ &=&\sum_{k=2}^\infty a_{k-1} \sum_{i=2^{k-1}}^{2^{k}-1} e_i -\lambda \sum_{k=1}^\infty a_k\sum_{i=2^{k-1}}^{2^k-1}e_i\nonumber\\ &=&
\sum_{k=1}^\infty\big ({{T}_\lambda} a\big)_k\sum_{i=2^{k-1}}^{2^{k}-1} e_i = S{{{T}_\lambda}a}.
\label{eq16}
\end{eqnarray}
%

\begin{prop}\label{Proposition 1}
Let $X$ be a symmetric sequence space and $\lambda>0$. Then, we have: 

(i) ${{D}_\lambda} $ is closed in ${X}$ if and only if ${{T}_\lambda}$ is closed in $E_X$;

(ii) ${{D}_\lambda} $ is an isomorphism in ${X}$ if and only if the operator ${{T}_\lambda}$ is an isomorphism in $E_X$.
\end{prop}

\begin{proof}
First, we observe that both operators ${{T}_\lambda}$ and ${{D}_\lambda} $ are injective for every $\lambda$. Indeed, for instance, the equation ${T}_\lambda a=0$, with $a=(a_n)$, means that $a_1=0$ and $a_{n-1}=\lambda a_n$ if $n\ge 2$, whence $a_n=0$ for all $n\in\mathbb{N}$. Therefore, it suffices to prove only (i). 

Suppose first that the operator ${{D}_\lambda} $ is closed in the space $X$. Let $a^n=(a^n_k)_{k=1}^\infty\in {E}_X$, $n=1,2,\dots$, and ${{T}_\lambda}a^n\to b=(b_k)$ in ${E}_X$. Since 
$$
{{T}_\lambda}a^n=-\lambda a^n_1e_1+\sum_{k=2}^\infty(a^n_{k-1}-\lambda a^n_{k})e_{k}$$
and ${{T}_\lambda}a^n\to b$ coordinate-wise as $n\to\infty$, we see that $b={{T}_\lambda}a$, where $a=(a_k)_{k=1}^\infty$, $a_k:=\lim_{n\to\infty} a^n_k$, $k\in\fg$. It remains to show that $a\in E_X$.


By hypothesis, the operator $D_\lambda$ is closed in $X$ (and hence it is an isomorphic mapping). Denote by $D_\lambda^{-1}$ the left converse operator to $D_\lambda$, defined on the subspace ${\rm Im}\,D_\lambda$ of $X$.  In view of \eqref{eq16}, it holds
$$
D_{\lambda}{S}{a^n}={S}{T}_\lambda{a^n}\;\;\mbox{for all}\;\;n=1,2,\dots,$$
whence
$$
{S}{a^n}=D_\lambda^{-1}{S}{T}_\lambda{a^n},\;\;n=1,2,\dots$$
Then, since the convergence ${{T}_\lambda}a^n\to b$ in ${E}_X$ implies that $S{{T}_\lambda}a^n\to Sb$ in $X$, we get: ${S}{a^n}\to D_\lambda^{-1}Sb$ in $X$. On the other hand, it is obvious that ${S}{a^n}\to Sa$ coordinate-wise as $n\to\infty$. Hence, $Sa=D_\lambda^{-1}Sb\in X$, which implies that $a\in E_X$, as we wish.



\smallskip
Let us prove the converse. To get a contradiction, assume that the operator $D_{\lambda}$ is not closed. Then, there exists a 
sequence $\{x^{(n)}\}\subset {X}$ with the following properties:
\begin{equation}
\label{eq19.1}
\|x^{(n)}\|_{{X}}=1,\;\;n=1,2,\dots,\;\;\mbox{and}\;\;\|
D_{\lambda}x^{(n)}\|_{{X}}\to 0.
\end{equation}
Since ${X}$ is separable or has the Fatou property, we have
$$
\|D_{\lambda}x^{(n)}\|_X\ge \|D(x^{(n)})^*-\lambda (x^{(n)})^*\|_X=\|D_{\lambda}(x^{(n)})^*\|_X$$
(see, e.g., \cite[Theorems II.4.9, II.4.10 and Lemma~II.4.6]{KPS} or \cite[\S\,3.7]{BSh}).
Consequently, we may assume that each of the sequences $x^{(n)}$, $n\in\fg$, is nonnegative and nonincreasing. 

Observe that the operators $Q$ (defined by \eqref{projection Q}) and $D_\lambda$ commute. Indeed, since $D=\tau_1\sigma_2$, we have
\begin{eqnarray*}
QD_{\lambda}x &=& \sum_{k=1}^\infty
2^{-k+1}\sum_{j=2^{k-1}}^{2^{k}-1} (D_{\lambda}x)_j \sum_{i=2^{k-1}}^{2^{k}-1} e_i\\
&=& \sum_{k=1}^\infty 2^{-k+1}\Big(\sum_{j=2^{k-1}}^{2^{k}-1} x_j \sum_{i=2^{k}}^{2^{k+1}-1} e_i-
\lambda\sum_{j=2^{k-1}}^{2^{k}-1}x_j\sum_{i=2^{k-1}}^{2^{k}-1} e_i\Big)\\
&=& 
D_{\lambda}\Big(\sum_{k=1}^\infty 2^{-k+1}\sum_{j=2^{k-1}}^{2^{k}-1}x_j\sum_{i=2^{k-1}}^{2^{k}-1} e_i\Big)=D_{\lambda}Qx.
\end{eqnarray*}
Therefore, noting that in our notation $Qx = Sa_x$, where
$$
a_x:=\left(2^{-k+1}\sum_{j=2^{k-1}}^{2^{k}-1} x_j\right)_{k=1}^\infty,$$
by \eqref{eq16}, we get for all $x\in X$
$$
S{{T}_\lambda}a_x=D_{\lambda}Sa_x=D_{\lambda}Qx=QD_{\lambda}x.$$
Since the projection $Q$ is bounded on ${X}$, then substituting $x^{(n)}$, $n=1,2,\dots$ for $x$ into the latter formula and taking   the limit as $n\to\infty$, by  \eqref{eq19.1}, we get 
\begin{equation}
\label{eq19.1aa}
\lim_{n\to\infty}\|{{T}_\lambda} a_{x^{(n)}}\|_{{E_X}}=0.
\end{equation}

On the other hand, by the definition of $E_X$ and the monotonicity of each sequence $x^{(n)}$,  
$$
\|a_{x^{(n)}}\|_{{E_X}}= \Big\|\sum_{k=1}^\infty 2^{-k+1}\sum_{j=2^{k-1}}^{2^{k}-1} x^{(n)}_j\sum_{i=2^{k-1}}^{2^{k}-1} e_i\Big\|_{{X}}\ge \Big\|\sum_{k=1}^\infty
x^{(n)}_{2^{k}-1}\sum_{i=2^{k-1}}^{2^{k}-1} e_i\Big\|_{{X}}.
$$
Consequently, since 
\begin{equation}
\label{eq19.1a}
(\sigma_{1/2}x^{(n)})_{m}=\frac12\left(x^{(n)}_{2m-1}+x^{(n)}_{2m}\right),\;\;m\in\fg,
\end{equation}
it follows that
\begin{eqnarray*}
\|a_{x^{(n)}}\|_{{E_X}}&\ge& \Big\|\sum_{k=1}^\infty
(\sigma_{1/2}x^{(n)})_{2^{k-1}}\sum_{i=2^{k-1}}^{2^{k}-1} e_i\Big\|_{{X}}\\ &\ge& \Big\|\sum_{i=1}^\infty
(\sigma_{1/2}x^{(n)})_{i}e_i\Big\|_{{X}}=\|\sigma_{1/2}x^{(n)}\|_X.
\end{eqnarray*}
Moreover, taking into account that $X$ is a symmetric space and $x^{(n)}\ge 0$, in view of \eqref{eq19.1a}, we have
\begin{eqnarray*}
\|x^{(n)}\|_X&\le&\Big\|\sum_{k=1}^\infty (x^{(n)})_{2k-1}e_{2k-1}\Big\|_X+\Big\|\sum_{k=1}^\infty (x^{(n)})_{2k}e_{2k}\Big\|_X\\&=&
\Big\|\sum_{k=1}^\infty (x^{(n)})_{2k-1}e_{k}\Big\|_X+\Big\|\sum_{k=1}^\infty (x^{(n)})_{2k}e_{k}\Big\|_X\\&\le& 2\max\left(\|(x^{(n)})_{2k-1}\|_X,\|(x^{(n)})_{2k}\|_X\right)\le 4\|\sigma_{1/2}x^{(n)}\|_X.
\end{eqnarray*}
Combining the last inequalities together with \eqref{eq19.1} implies that
$$
\|a_{x^{(n)}}\|_{{E_X}}\ge\frac14,\;\;n\in\fg,$$
and hence, by \eqref{eq19.1aa}, the operator ${T}_\lambda$ fails to be an isomorphic embedding in the space $E_X$. Since it is injective, this means that ${T}_\lambda$ is not closed in $E_X$, which contradicts the assumption.
\end{proof}

The following result is an immediate consequence of Proposition \ref{Proposition 1}. 

\begin{cor}
\label{cor:basic1}
Let $X$ be a symmetric sequence space. Then the operators ${{D}}$ in $X$ and ${\tau}_1$ in $E_X$ have the same set of positive approximate eigenvalues.
\end{cor}

\def\vr{{\cal T}}
\def\gh{{\mathbb R}}
\def\hj{{\mathbb R}}
\def\bc{\sum_{k=0}^\infty}
\def\cd{{\mathbb N}}
\def\ef{E}
\def\fg{{\mathbb N}}
\def\ob{\sum_{k=1}^\infty}
\def\hi{\sum_{k=1}^{\infty}a_kr_{k}}
\def\ji{{\cal J}}
\def\kl{(a_k)_{k=1}^{\infty}}
\def\mn{{T}_{\lambda}}
\def\op{{\rm Im}\,{T}_{\lambda}}
\def\pq{{\mathbb R}}
\def\zw{{\rm Ker}\,f_\lambda}
\def\qr{{\mathbb N}}
\def\rs{X_{\theta,q}}
\def\st{{Y_{\theta,q}}}

\section{Approximative eigenvalues of the shift operator in Banach sequence lattices.}

Suppose $E$ is a Banach sequence lattice such that the shift operator $\tau_1(a_k)=(a_{k-1})$ and its inverse $\tau_{-1}(a_k)=(a_{k+1})$  are bounded in $E$. Let $s_k:=\|e_k\|_E$, where $e_k$, $k=1,2,\dots$, are elements of the unit vector basis. 

Next, we suppose that the shift exponents $k_-(E)$ and $k_+(E)$ can be calculated when the shift operators $\tau_n(a_k)=(a_{k-n})$, $n\in\mathbb{Z}$, are restricted to the set $\{e_k\}_{k=1}^\infty$, or, more explicitly, it will be assumed that 
\begin{equation}
\label{equa16b} 
k_+(E)=\lim_{n\to\infty}\left(\sup_{k>n}\frac{s_k}{s_{k-n}}\right)^{1/n}\;\;\mbox{and}\;\;k_-(E)=\lim_{n\to\infty}\left(\sup_{k\in\cd}\frac{s_k}{s_{n+k}}\right)^{1/n}.
\end{equation}
Hence, $1/k_-(E)\le k_+(E)$.  

As we show in this section, then we are able to identify, in terms of the exponents $k_-(E)$ and $k_+(E)$, the set of all parameters $\lambda>0$, for which the operator $\mn=\tau_1-\lambda I$ (as usual, $I$ is the identity in $E$) is an isomorphic embedding in $E$.

%

\begin{prop}
\label{L2-new}
Suppose a separable Banach sequence lattice $E$ satisfies assumption  \eqref{equa16b}. Then, for every $\lambda>0$ the following conditions are equivalent:

(i) the operator $\mn$ is an isomorphic mapping in ${\ef}$;

(ii) the operator $\mn$ is closed in ${\ef}$;

(iii) $\lambda\in (0,1/k_-(E))\cup (k_+(E),\infty).$

Moreover, if $\lambda\in (k_+(E),\infty)$, then $\op ={\ef};$ if
$\lambda\in (0,1/k_-(E))$, then $\op$ is the closed subspace of $E$ of codimension $1$ consisting of all $(a_k)_{k=1}^\infty\in{\ef}$ with
\begin{equation}
\label{equa15a} 
\ob \lambda^k a_k=0.
\end{equation}
\end{prop}
\begin{proof}
First of all, observe that the operator $\mn$ is injective for every $\lambda$. Consequently, the equivalence of conditions (i) and (ii) is obvious.

Let us determine the possible form of the subspace $\op$. From the equality
$$
\mn\left(\sum_{i=1}^{n} \lambda^{n-i}e_i\right)=e_{n+1}-\lambda^{n}e_1$$
it follows that $e_{n+1}-\lambda^{n}e_1\in\op$ for all $n\in\cd.$ Hence, if $f_\lambda$  is a linear functional vanishing at $\op$, we have $f_\lambda (e_{n+1})=\lambda^{n}f_\lambda (e_{1})$, $n\in\cd$. Therefore, we may assume that $f_\lambda$  corresponds to the sequence $(\lambda^{n-1})_{n=1}^\infty.$ Observe that if $a=(a_k)_{k=1}^\infty\in c_{00}$, then $f_\lambda({T}_\lambda a)=0$. Indeed, for each $k\in \mathbb{N}$, $k\ge 2$,
$$ 
{T}_\lambda e_k ={T}_\lambda \left(\sum_{i=1}^{k} \lambda^{k-i}e_i-\lambda\sum_{i=1}^{k-1} \lambda^{k-1-i}e_i\right)=e_{k+1}-\lambda^{k}e_1-\lambda(e_{k}-\lambda^{k-1}e_1)=e_{k+1}-\lambda e_{k}.
$$
Since the last equation holds also for $k=1$ and $f_\lambda(e_{k+1}-\lambda e_{k})=0$, $k=1,2,\dots$, the claim is proved. 

By assumption, $E$ is separable. Hence, the dual space $E^*$ coincides with the K\"{o}the dual $E'$ (see Section \ref{prel1}), whence the condition $f_\lambda\in{\ef} ^*$ is equivalent to the fact that 
\begin{equation}
\label{eq16a} 
\ob \lambda^{k}a_k <\infty,\;\;\mbox{for all}\;\;(a_k)_{k=1}^\infty\in E.
\end{equation}
Moreover, if \eqref{eq16a} holds, then the Hahn-Banach theorem combined with the preceding reasoning implies that
\begin{equation}
\label{eq16b} 
\overline{\op}=\zw.
\end{equation}
On the other hand, if $f_\lambda\not\in{\ef} ^*$, we have
\begin{equation}
\label{eq16c} 
\overline{\op}={\ef}.
\end{equation}

If $\lambda\in (k_+(E),\infty)$, then the fact that the spectral radius of $\tau_1$ is equal to $k_+(E)$ immediately implies that the operator $\mn$ is an isomorphism of ${\ef}$ onto ${\ef}$. 

Let us consider now the case when $\lambda\in (0,1/k_-(E))$. We show that the functional $f_\lambda$ is bounded on ${\ef}$ and the subspace $\op$ is closed in $E$. Once we prove this, by \eqref{eq16b}, ${\op}={\zw}$ and so being injective $\mn$  is an isomorphic mapping of ${\ef}$ onto the closed subspace $\zw$ consisting of all $(a_k)_{k=1}^\infty\in{\ef}$ satisfying \eqref{equa15a}.

Choose $\eta\in (\lambda,1/k_-(E))$. Then, $1/\eta > k_-(E)$ and hence, in view of \eqref{equa16b}, there is $C > 0$ such that  
$$
 \sup_{k\ge 0}\frac{s_k}{s_{k+n}}\le C\eta^{-n},\;\;n=1,2,\dots,$$
whence
$$
s_n ^{-1}\le s_1^{-1}C\eta^{-n},\;\;n=1,2,\dots.$$
On the other hand, since $E$ is a Banach lattice, for every $a=(a_k)_{k=1}^\infty\in E$ we have $|a_k|\|e_k\|_E\le\|a\|_E$, i.e., $|a_k|\le\|a\|_Es_k^{-1}$. Combining the last inequalities, we get
$$
\sum_{n=1}^\infty \lambda^{n}a_n\le \|a\|_E\sum_{n=1}^\infty \lambda^{n}s_n^{-1}\le s_1^{-1}C\|a\|_E\sum_{n=1}^\infty (\lambda/\eta)^{n}<\infty.$$
Thus, condition \eqref{eq16a} is fulfilled for each $a=(a_k)_{k=1}^\infty\in E$, yielding $f_\lambda\in{\ef} ^*.$ 

In order to prove that $\op$ is closed, we represent ${\ef}$ as follows:
$$
{\ef}=E_1+E_\infty,$$
where $E_1=[\{e_1\}]_{E}$ and $E_\infty=\,[\{e_n\}_{n\ge 2}]_{E}$
$([A]_{E}$ is the closed linear span of a set $A$ in the space $\ef).$
Since $\op=\mn(E_1)+\mn(E_\infty)$ and the space $\mn(E_1)$ is one-dimensional, it suffices to prove only that the set $\mn(E_\infty)$ is closed.

One can easily check that
\begin{equation}
\label{equa16d} 
\mn(E_\infty)=-\lambda \tau_1(\tau_{-1}-\lambda^{-1} I)(E_\infty),
\end{equation}
Since the spectral radius of $\tau_{-1}$ is equal to $k_-(E)$, the operator $\tau_{-1}-\lambda^{-1} I$ maps isomorphically ${\ef}$ onto ${\ef}$ whenever $0<\lambda<1/k_-(E)$. Consequently, since 
$E_\infty$ is closed in $E$, the set $(\tau_{-1}-\lambda^{-1} I)(E_\infty)$ is closed in $E$ as well. Then, in view of the fact that $\tau_1$ is an
isomorphic embedding in $E$, from \eqref{equa16d} it follows that the subspace $\mn(E_\infty)$ is closed.

Thus, in the case $0 <\lambda<1/k_-(E)$ the operator $\mn$ maps the space $\ef$ isomorphically onto the subspace of codimension $1$ consisting of all
$(a_k)\in{\ef}$ satisfying \eqref{equa15a}.

It remains to prove implication $(i) \Longrightarrow (iii)$, i.e., that the operator $\mn:\,{\ef}\to\ef$ fails to be an isomorphic embedding whenever $1/k_-(E)\le \lambda\le k_+(E)$.

Suppose first that $1/k_-(E)<\lambda<k_+(E)$. To the contrary, assume that there exists $c > 0$ such that for all $x\in{\ef}$
\begin{equation}
\label{equa18} 
\|\mn x\|_{\ef}\ge c\|x\|_{\ef}.
\end{equation}
Let $n,k\in\qr$, $k>n$, be arbitrary for a moment (they will be fixed later). We put
$a:=(I+\lambda^{-1}{\tau_1}+\dots+\lambda^{-n}{\tau_1}^n)^2e_{k-n}.$
A direct calculation shows that $a\ge n\lambda^{-n}e_{k},$ whence
\begin{equation}
\label{equa19} 
\|a\|_{\ef}\ge n\lambda^{-n}s_{k}.
\end{equation}

Now, we estimate the norm $\|\mn ^2a\|_{\ef}$ from above. First,
\begin{multline*}
\mn ^2(I+\lambda^{-1}{\tau_1}+\dots +\lambda^{-n}{\tau_1}^n)^2=\\
=\lambda({\tau_1}-\lambda I)(\lambda^{-(n+1)}{\tau_1}^{n+1}-I)(I+\lambda^{-1}{\tau_1}+\dots +\lambda^{-n}{\tau_1}^n)=\\
=\lambda^{2}I-2\lambda^{-(n-1)}{\tau_1}^{n+1}+\lambda^{-2n}{\tau_1}^{2n+2}.
\end{multline*}
Consequently,
$$
\mn ^2a=\lambda^{2}e_{k-n}-2\lambda^{-(n-1)}e_{k+1}+\lambda^{-2n}e_{n+k+2},$$
and from the triangle inequality it follows
\begin{eqnarray*}
\|\mn ^2a\|_{\ef}&=&\lambda^{2}s_{k-n}+2\lambda^{-(n-1)}s_{k+1}+\lambda^{-2n}s_{n+k+2}\\
&\le& \lambda^2s_{k-n}+2\lambda^{-(n-1)}\|{\tau_1}\|_{E}s_{k}+\lambda^{-2n}\|{\tau_1}\|_{E}^2s_{n+k}.
\end{eqnarray*}
Hence,
$$
\|\mn ^2a\|_{\ef}-2\lambda\|{\tau_1}\|_{E}\cdot \lambda^{-n}s_{k}\le 2\max(\lambda^{2},\|{\tau_1}\|_{E}^2)\max(s_{k-n},\lambda^{-2n}s_{n+k}).$$
Let us observe that \eqref{equa18} and \eqref{equa19} yield
$$
\|\mn ^2a\|_{\ef}\ge c^2n\lambda^{-n}s_{k}.$$
Therefore, choosing $n\in\qr$ so that 
$$
c^2n>2\lambda\|{\tau_1}\|_{E}+2\max(\lambda^{2},\|{\tau_1}\|_{E}^2),$$
from the preceding inequality we get
$$
\lambda^{-n}s_{k}<\max(s_{k-n},\lambda^{-2n}s_{n+k}),$$
or, equivalently,
\begin{equation}
\label{equa20} 
\nu _{k}<\max(\nu _{k-n},\nu _{n+k})\;\;\mbox{for all}\;\;k>n,
\end{equation}
where $\nu_n:=\lambda^{-n}s_n$.

By assumption, $\lambda <k_+(E)$. Therefore, by \eqref{equa16b}, for the already chosen number $n\in\qr$ we can find $k\in\cd$ such that $k>n$ and
$s_{k}>\lambda^{n}s_{k-n}$, i.e., $\nu _{k}>\nu _{k-n}$. Hence,  from \eqref{equa20} it follows that $\nu _{k+n}>\nu _{k}$. Substituting $k + n$ for $k$ in \eqref{equa20} and taking into account that $\nu _{k+n}>\nu _{k}$, we obtain $\nu _{k+2n}>\nu _{k+n}$. Proceeding in the same way, we conclude that the sequence $(\nu _{k+rn})_{r=0}^\infty$, with the above $n,k\in\mathbb{N}$, is increasing.

Let $j\ge k$ and $m\ge n.$ We find $1\le r_1\le r_2$ such that
$$
k+(r_1-1)n \le j \le k+r_1n\;\;\mbox{and}\;\;k+r_2n \le j+m \le k+(r_2+1)n.$$
Setting $C_1:=\max_{l=1,2,\dots,n}\|\tau_{-l}\|_{E\to E}$, we have
$$
s_j\le C_1s_{k+r_1n}\;\;\mbox{and}\;\;s_{k+r_2n}\le C_1s_{j+m}.$$
Therefore,
\begin{equation}
\label{equa20a} 
\frac{s_{j+m}}{s_j}\ge C_1^{-2}\frac{s_{k+r_2n}}{s_{k+r_1n}}=
C_1^{-2}\frac{\lambda^{r_2n}\nu _{k+r_2n}}{\lambda^{r_1n}\nu _{k+r_1n}}\ge \lambda^{n(r_2-r_1)}.
\end{equation}

If $\lambda\ge 1$, then in view of the inequality $m-2n\le(r_2-r_1)n,$ we deduce from \eqref{equa20a} that
$$
\frac{s_{j+m}}{s_j}\ge C_1^{-2}\lambda^{m-2n}\;\;\mbox{for all}\;j\ge k\;\mbox{and}\;m\ge n,$$
or
$$
\sup_{j\ge k}\frac{s_j}{s_{j+m}}\le C_2\lambda^{-m}\;\;\mbox{for all}\;m\in \fg.$$
Moreover, for all $1\le j<k$ we have
$$
\frac{s_j}{s_{j+m}}\le\frac{s_j}{s_{k}}\cdot\frac{s_{k+m}}{s_{j+m}}\cdot\frac{s_{k}}{s_{k+m}}\le C_3\frac{s_k}{s_{k+m}},$$
where $C_3:=\max_{l=\pm 1,\pm 2,\dots,\pm k}\|\tau_{l}\|_{E\to E}$.
Combining the last estimates, we infer
$$
\sup_{j\in\mathbb{N}}\frac{s_j}{s_{j+m}}\le C\lambda^{-m}$$
for some constant $C>0$ and all $m\in \fg$.
Hence, in view of \eqref{equa16b}, we conclude that $k_-(E)\le 1/\lambda$, which contradicts the assumption.

If $0<\lambda<1$, then from the inequality $m\ge(r_2-r_1)n$ and estimate \eqref{equa20a} it follows that
$$
\frac{s_{j+m}}{s_j}\ge C_1^{-2}\lambda^{m}\;\;\mbox{for all}\;j\ge k\;\mbox{and}\;m\ge n.$$
Then, reasoning precisely in the same way, we again obtain that $k_-(E)\le 1/\lambda$. As a result, the operator ${T}_{\lambda}$ fails to be an isomorphic embedding in ${\ef}$ for 
$\lambda\in (1/k_-(E),k_+(E))$.

Finally, we need to check that neither ${T}_{k_+(E)}$ nor ${T}_{1/k_-(E)}$ is an isomorphic embedding in ${\ef}$. 

According to \eqref{eq16b} and \eqref{eq16c}, if  ${T}_\lambda$ is an isomorphic embedding in ${\ef}$, then it is both a Fredholm operator of index $0$ or $-1$. In particular, it was proved that ${T}_\lambda$ is a Fredholm operator of index $0$ (resp. $-1$) if $\lambda\in(k_+(E),\infty)$ (resp. $\lambda\in (0,1/k_-(E))$). Moreover, we know that in the case when $1/k_-(E)<k_+(E)$ the operator ${T}_\lambda$ fails to be an isomorphic embedding in $E$ for $\lambda\in (1/k_-(E),k_+(E))$. Therefore, since the set of all Fredholm operators is open in the space of all bounded linear operators on ${\ef}$ with respect to the topology generated by the operator norm (see, e.g., \cite[Theorem~III.21]{KG}), we conclude that each of the operators ${T}_{k_+(E)}$ nor ${T}_{1/k_-(E)}$ may not be isomorphic. The same result follows if $1/k_-(E)=k_+(E)$, because the set of Fredholm operators with a fixed index is also open (see, e.g., \cite[Theorem~III.22]{KG}). 



\end{proof}

\begin{cor}
\label{cor-new}
Let $X$ be a separable symmetric sequence space of fundamental type.
Then, for every $\lambda>0$ the following conditions are equivalent:

(i) the operator $D$ is an isomorphic mapping in $X$;

(ii) the operator $D$ is closed in $X$;

(iii) $\lambda\in (0,2^{\alpha_X})\cup (2^{\beta_X},\infty).$

Hence, the set of all positive approximate eigenvalues of the operator $D$ coincides with the interval $[2^{\alpha_X},2^{\beta_X}]$.
\end{cor}

\begin{proof}
Show that the Banach lattice $E_X$ satisfies assumption   \eqref{equa16b}. Recall that
$$
M_X^\infty(2^n)=\sup_{m\in\fg}\frac{\phi_X(2^nm)}{\phi_X(m)},\;\;n\in\fg
$$
(see Section \ref{prel4}). Choosing for each $m\in\fg$ a nonnegative integer $k$ so that $2^k\le m<2^{k+1}$, we have $\phi_X(2^k)\le \phi_X(m)\le 2\phi_X(2^k)$. Therefore, from the preceding formula it follows
$$
\sup_{k\in\fg}\frac{\phi_X(2^{k+n})}{\phi_X(2^k)}\le M_X^\infty(2^{n})
\le 2\sup_{k\in\fg}\frac{\phi_X(2^{k+n})}{\phi_X(2^k)},\;\;n\in\fg,$$
whence
$$
M_X^\infty(2^{n})\asymp \sup_{k\in\fg}\frac{\|\sum_{i=1}^{2^{k+n}}e_i\|_X}{\sum_{i=1}^{2^{k}}e_i\|_X},\;\;n\in\fg.$$
On the other hand, by the definition of $E_X$, we have 
$$
s_k=\|e_k\|_{E_X}=\Big\|\sum_{i=2^{k-1}}^{2^k-1}e_i\Big\|_X,\;\;k\in\fg,$$ which combined with the fact that $X$ is a symmetric space, implies  
$$
\sup_{k\in\fg}\frac{\|\sum_{i=1}^{2^{k+n}}e_i\|_X}{\sum_{i=1}^{2^{k}}e_i\|_X}=\sup_{k>n}\frac{\|\sum_{i=2^{k-1}}^{2^k-1}e_i\|_X}{\|\sum_{i=2^{k-n-1}}^{2^{k-n}-1}e_i\|_X}=\sup_{k>n}\frac{s_k}{s_{k-n}}.$$
Thus,
$$
M_X^\infty(2^{n})\asymp \sup_{k>n}\frac{s_k}{s_{k-n}},\;\;n\in\fg,$$
and hence, by the definition of $\nu_X$ and the assumption that $X$ is a space of fundamental type, we have
$$
\beta_X=\nu_{X}=\lim_{n\to\infty}\frac{1}{n}\log_2 \sup_{k>n}\frac{s_k}{s_{k-n}}.$$
Applying now Lemma \ref{dilation}, we obtain the first equality in \eqref{equa16b}. The proof of the second is quite similar, and we skip it. 

Since $E_X$ is separable together with $X$, all the assertions of the corollary follow now from Propositions \ref{L2-new} and \ref{Proposition 1} and Lemma \ref{dilation}.
\end{proof}

\smallskip
From the last corollary and Theorem \ref{Theorem 1a} it follows a complete description of the set ${\mathcal F}(X)$ for every separable symmetric sequence space $X$ of fundamental type.

\begin{theor}\label{Theorem 4}
Let $X$ be a separable symmetric sequence space of fundamental type. Then, the following conditions are equivalent:

(i) $\ell^p$ is symmetrically block finitely represented in the unit vector basis $\{e_k\}$;

(ii) $\ell^p$ is symmetrically finitely represented in $X$;

(iii) $\ell^p$ is crudely symmetrically finitely represented in $X$;

(iv) $p\in [1/\beta_{X},1/\alpha_{X}]$, where $\alpha_{X}$ and $\beta_{X}$ are the Boyd indices of $X$.
\end{theor}


%

\begin{remark}
By Theorem {\rm\ref{Th: Krivine3}}, if $p=1/\alpha_X$ or $1/\beta_X$, the space $\ell^p$ is symmetrically finitely represented in {\it every} (not necessarily of fundamental type) symmetric sequence space $X$.
\end{remark}

\vskip0.5cm

\section{Symmetric finite representability of $\ell^p$ in Lorentz and Orlicz spaces.}
\label{Lor and Orl}

Results obtained allow us to find the set ${\mathcal F}({X})$ if $X$ is a Lorentz or a separable Orlicz sequence space.

\subsection{Lorentz spaces.}

\def\vr{{\cal T}}
\def\gh{{\cal K}}
\def\hj{{\mathbb R}}
\def\bc{\sum_{k=0}^\infty}
\def\cd{{\mathbb N}}
\def\ef{l_q(\mu)}
\def\fg{{\mathbb N}}
\def\ob{\sum_{k=1}^\infty}
\def\hi{\sum_{k=1}^{\infty}a_kr_{k}}
\def\ji{{\cal J}}
\def\kl{(a_k)_{k=1}^{\infty}}
\def\mn{T_{\theta}}
\def\op{{\rm Im}T_{\theta}}
\def\pq{{\mathbb C}}
\def\zw{{\rm Ker}f_\theta}
\def\qr{{\mathbb N}}
\def\rs{X_{\theta,q}}
\def\st{{Y_{\theta,q}}}

Let $1\le q<\infty$, and let $\{w_k\}_{k=1}^\infty$ be a nonincreasing sequence of positive numbers.
Recall that the Lorentz space $\lambda_q(w)$ consists of all sequences $a=(a_k)_{k=1}^\infty$ such that
$$
\|a\|_{\lambda_q(w)}=\Big(\sum_{k=1}^\infty (a_k^*)^qw_k^q\Big)^{1/q}<\infty.$$

Since $\lambda_q(w)$ is a separable symmetric space of fundamental type (see Section \ref{prel3}), applying Theorem \ref{Theorem 4} implies

\begin{theor}\label{Theorem 4a}
For every Lorentz sequence space $\lambda_q(w)$ the following conditions are equivalent:

(i) $\ell^p$ is symmetrically block finitely represented in the unit vector basis $\{e_k\}$ of $\lambda_q(w)$;

(ii) $\ell^p$ is symmetrically finitely represented in $\lambda_q(w)$;

(iii) $\ell^p$ is crudely symmetrically finitely represented in $\lambda_q(w)$;

(iv) $p\in [1/\beta_{\lambda_q(w)},1/\alpha_{\lambda_q(w)}]$, where the Boyd indices $\alpha_{\lambda_q(w)}$ and $\beta_{\lambda_q(w)}$ are defined by \eqref{Boyd ind Lor} and \eqref{Boyd ind Lor1}.
\end{theor}

Suppose additionally that 
\begin{equation}
\label{equa22} 
\lim_{n\to\infty}\left(\sup_{k=0,1,\dots}\frac{w_{2^k}}{w_{2^{k+n}}}\right)^{1/n}<2^{1/q}.
\end{equation}
By this assumption, we prove that the Banach lattice $E_{\lambda_q(w)}$ coincides with the weighted space $\ef$ equipped with the norm
$$
\|a\|_{\mu,q}:=\Big(\sum_{k=1}^\infty |a_{k}|^q\mu_k^q\Big)^{1/q},$$
where $\mu_k=2^{(k-1)/q}w_{2^{k-1}}$, $k=1,2,\dots$. Moreover, in this case the Boyd indices of $\lambda_q(w)$ can be calculated by the simpler (than \eqref{Boyd ind Lor} and \eqref{Boyd ind Lor1}) formulae:
\begin{equation}
\label{equa22simple} 
\alpha_{\lambda_q(w)}=-\lim_{n\to\infty}\frac1n\log_2\sup_{k=0,1,\dots}\frac{w_{2^k}}{w_{2^{k+n}}}\;\;\mbox{and}\;\;\beta_{\lambda_q(w)}=\lim_{n\to\infty}\frac1n\log_2\sup_{k>n}\frac{w_{2^k}}{w_{2^{k-n}}}.\end{equation}
 
Show first that 
\begin{equation}
\label{equa21} 
\|x\|_{\lambda_q(w)}\asymp\Big\|\sum_{k=0}^\infty x_{2^{k}}^*e_{k+1}\Big\|_{\ef}.
\end{equation}
Indeed, assuming for simplicity of notation that $x=x^*$, we get
$$ 
\|x\|_{\lambda_q(w)}^q=\sum_{k=1}^\infty\sum_{i=2^{k-1}}^{2^{k}-1} x_i^qw_i^q\le\sum_{k=1}^\infty 2^{k-1}x_{2^{k-1}}^qw_{2^{k-1}}^q= \sum_{k=0}^\infty x_{2^{k}}^q2^kw_{2^{k}}^q.
$$
On the other hand, one has 
$$ 
2\|x\|_{\lambda_q(w)}^q=x_1^qw_1^q+\sum_{k=1}^\infty x_{2^{k}}^q\sum_{i=2^{k-1}}^{2^{k}-1} w_i^q\ge x_1^qw_1^q+\sum_{k=1}^\infty x_{2^{k}}^q2^{k-1}w_{2^{k}}^q\ge\frac12 \sum_{k=0}^\infty x_{2^{k}}^q2^{k}w_{2^{k}}^q,
$$
and so \eqref{equa21} follows.

Moreover, since $\|e_k\|_{\ef}=\mu_k$, $k=1,2,\dots$, one can easily see that
$$
\|\tau_{-n}\|_{\ef\to \ef}=\sup_{k=1,\dots}\frac{\mu_k}{\mu_{k+n}}=2^{-n/q}\sup_{k=0,1,\dots}\frac{w_{2^k}}{w_{2^{k+n}}},\;\;n\in\mathbb{N},$$
and
$$
\|\tau_{n}\|_{\ef\to \ef}=\sup_{k>n}\frac{\mu_k}{\mu_{k-n}}=2^{n/q}\sup_{k>n}\frac{w_{2^k}}{w_{2^{k-n}}},\;\;n\in\mathbb{N}.$$
From the first of these equalities and \eqref{equa22} it follows 
$$
k_-(\ef)=\lim_{n\to\infty}\|\tau_{-n}\|_{\ef\to \ef}^{1/n}<1.$$
This relation combined with equivalence \eqref{equa21} indicates that the Banach lattice $\ef$ satisfies all the conditions of Proposition \ref{prop:latttice and symm}, and hence we have $\ef=E_{\lambda_q(w)}$ (with equivalence of norms). Moreover, by the above formulae for the norms of the shift operators in $\ef$ and Lemma \ref{dilation}, we get \eqref{equa22simple}.

\subsection{Orlicz spaces.}

Given an Orlicz function $N$, the Orlicz sequence space $l_N$  consists of all sequences $a=(a_k)_{k=1}^\infty$, for which the norm
$$
\|a\|_{l_N}=\inf\left\{u>0:\,\sum_{k=1}^\infty N\Big(\frac{|a_k|}{u}\Big)\le 1\right\}$$
is finite. Recall that $l_N$ is separable if and only if the function $N$ satisfies the $\Delta_2$-condition at zero (see Section \ref{prel2}). It can be easily verified that the latter condition is equivalent to the fact that the lower Boyd index of $l_N$ (see \eqref{equa24} or \eqref{equa24om}) is positive. Since $l_N$ is of fundamental type (cf. Section \ref{prel3} or \cite{Boyd}), we arrive at the following characterization of the set of $p$ such that $\ell^p$ is symmetrically finitely represented in Orlicz spaces.

\begin{theor}\label{Theorem 4b}
Let $N$ be an Orlicz function such that for some $r>0$
$$
\inf_{0<st\le 1}\frac{N(st)}{N(s)t^r}>0.$$
Then, the following conditions are equivalent:

(i) $\ell^p$ is symmetrically block finitely represented in the unit vector basis $\{e_k\}$ of the Orlicz space $l_N$;

(ii) $\ell^p$ is symmetrically finitely represented in $l_N$;

(iii) $\ell^p$ is crudely symmetrically finitely represented in $l_N$;

(iv) $p\in [1/\beta_{l_N},1/\alpha_{l_N}]$, where $\alpha_{l_N}$ and $\beta_{l_N}$ are the Boyd indices of $l_N$ (cf. \eqref{equa24} or 
 \eqref{equa24om}).
\end{theor}

%


In conclusion, we prove\footnote{Simultaneously we give one more independent proof of formulae  \eqref{equa24}.}  that $E_{l_N}=U_N$ (with equivalence of norms), where $U_N$ is the Banach sequence lattice equipped with the norm
$$
\|a\|_{U_N}:=\inf\left\{u>0:\,\sum_{k=1}^\infty 2^{k-1} N\Big(\frac{|a_k|}{u}\Big)\le 1\right\}.$$

We claim that
\begin{equation}
\label{equa26} 
\|a\|_{l_N}\asymp\Big\|\sum_{k=0}^\infty a_{2^{k}}^*e_{k+1}\Big\|_{U_N}.
\end{equation}
First, assuming again that $a=a^*$ and using the monotonicity of $N$, we obtain
\begin{equation}
\label{equa27} 
\sum_{k=1}^\infty N\Big(\frac{a_k}{u}\Big)=\sum_{k=1}^\infty\sum_{i=2^{k-1}}^{2^{k}-1} N\Big(\frac{a_i}{u}\Big)\le\sum_{k=1}^\infty 2^{k-1} N\Big(\frac{a_{2^{k-1}}}{u}\Big)=\sum_{k=0}^\infty 2^kN\Big(\frac{a_{2^{k}}}{u}\Big),
\end{equation}
and similarly in the opposite direction
\begin{equation}
\label{equa28} 
2\sum_{k=1}^\infty N\Big(\frac{a_k}{u}\Big)\ge N\Big(\frac{a_1}{u}\Big)+\sum_{k=1}^\infty 2^{k-1} N\Big(\frac{a_2^{k}}{u}\Big)\ge\frac12\sum_{k=0}^\infty 2^kN\Big(\frac{a_{2^{k}}}{u}\Big).
\end{equation}

Next, if $u > \|\sum_{k=0}^\infty a_{2^{k}}e_{k+1}\|_{U_N}$, from the 
definition of the latter norm and inequality \eqref{equa27} it follows
$$
\sum_{k=1}^\infty N\Big(\frac{a_k}{u}\Big)\le \sum_{k=0}^\infty 2^{k} N\Big(\frac{a_{2^k}}{u}\Big)\le 1,$$
which implies that $u\ge \|a\|_{l_N}$. Hence, $\|a\|_{l_N}\le\|\sum_{k=0}^\infty a_{2^{k}}e_{k+1}\|_{U_N}$.

Conversely, if $u >\|a\|_{l_N}$, then from \eqref{equa28} and the convexity of $N$ it follows
$$
\sum_{k=0}^\infty 2^{k} N\Big(\frac{a_{2^k}}{4u}\Big)\le \frac14\sum_{k=0}^\infty 2^{k} N\Big(\frac{a_{2^k}}{u}\Big)\le\sum_{k=1}^\infty N\Big(\frac{a_k}{u}\Big)\le 1,$$
whence $4u\ge \|\sum_{k=0}^\infty a_{2^{k}}e_{k+1}\|_{U_N}$. Consequently, $\|\sum_{k=0}^\infty a_{2^{k}}e_{k+1}\|_{U_N}\le 4\|a\|_{l_N}$, and so equivalence \eqref{equa26} is proved.

Further, we want to show that
\begin{equation}
\label{equa29} 
\|\tau_{-n}\|_{E_N\to E_N}\asymp \sup_{k=0,1,\dots}\frac{N^{-1}(2^{-k-n})}{N^{-1}(2^{-k})}\;\;\mbox{and}\;\;\|\tau_n\|_{E_N\to E_N}\asymp \sup_{k\ge n}\frac{N^{-1}(2^{-k+n})}{N^{-1}(2^{-k})},\;\;n\in\mathbb{N}.
\end{equation}

Observe first that for any $s\in (0,1]$ and $n\in\mathbb{N}$ 
\begin{equation}
\label{equa30} 
\frac{N^{-1}(2^{-n+1}s)}{N^{-1}(s)}\le 4A_n,
\end{equation}
where
$$
A_n:=\sup_{m=0,1,\dots}\frac{N^{-1}(2^{-m-n})}{N^{-1}(2^{-m})},\;\;n\in\mathbb{N}.$$
Indeed, choosing positive integer $m$ so that $s\in (2^{-m},2^{-m+1}]$, because the inverse function $N^{-1}$ is nondecreasing and concave, we get
$$
\frac{N^{-1}(2^{-n+1}s)}{N^{-1}(s)}\le\frac{N^{-1}(2^{-n-m+2})}{N^{-1}(2^{-m})}\le \frac{4N^{-1}(2^{-n-m})}{N^{-1}(2^{-m})}\le 4A_n,$$
and \eqref{equa30} follows.

Let us check now that for every $t>0$ such that $2^{n-1}N(t)\le 1$ it holds
\begin{equation}
\label{equa31} 
N\Big(\frac{t}{4A_n}\Big)\le 2^{n-1}N(t).
\end{equation}
Indeed, passing to the inverses, we get 
$$
t\le 4A_nN^{-1}(2^{n-1}N(t)),$$
and then, setting $s=2^{n-1}N(t)$,
$$
N^{-1}(2^{-n+1}s)\le 4A_nN^{-1}(s).$$
Since $0<s\le 1$, the latter inequality is a consequence of \eqref{equa30}, and thus \eqref{equa31} is proved.

Suppose $n\in\mathbb{N}$, $a=(a_k)_{k=1}^\infty\in U_N$ and $u>\|a\|_{U_N}$. Then,
$$
\sum_{k=1}^\infty 2^{k-1} N\Big(\frac{|a_k|}{u}\Big)\le 1.$$
Consequently, $2^{n-1}N(|a_{n+k}|/u)\le 1$ for each $k\in\fg$. 
Applying now inequality \eqref{equa31} for $t=|a_{n+k}|/u$, $k\in\fg$, we get
\begin{eqnarray*}
\sum_{k=1}^\infty 2^{k-1} N\Big(\frac{|(\tau_{-n}a)_{k}|}{4A_nu}\Big)&=&\sum_{k=1}^\infty 2^{k-1}N\Big(\frac{|a_{k+n}|}{4A_nu}\Big)\\&\le& \sum_{k=1}^\infty 2^{k+n-2} N\Big(\frac{|a_{k+n}|}{u}\Big)\\&\le&\sum_{k=1}^\infty 2^{k-1} N\Big(\frac{|a_{k}|}{u}\Big)\le 1,
\end{eqnarray*}
which implies that $\|\tau_{-n}a\|_{U_N}\le 4A_n$. Therefore, by the 
definition of $A_n$, we infer that 
$$
\|\tau_{-n}\|_{U_N\to U_N}\le 4\sup_{m=0,1,\dots}\frac{N^{-1}(2^{-m-n})}{N^{-1}(2^{-m})},\;\;n\in\mathbb{N}.$$
Since from the definition of $U_N$ it follows that $\|e_k\|_{U_N}=1/N^{-1}(2^{-k+1})$, $k=1,2,\dots$, the opposite inequality
$$
\|\tau_{-n}\|_{U_N\to U_N}\ge \sup_{m=0,1,\dots}\frac{N^{-1}(2^{-m-n})}{N^{-1}(2^{-m})},\;\;n\in\mathbb{N},$$
is obvious. Therefore, the first equivalence in \eqref{equa29} is  proved. The proof of the second equivalence in \eqref{equa29} is quite similar, and we skip it.

From \eqref{equa29} it follows that
$$
k_-(U_N)=\lim_{n\to\infty}\sup_{m=0,1,\dots}\frac{N^{-1}(2^{-m-n})}{N^{-1}(2^{-m})}\;\;\mbox{and}\;\;k_+(U_N)= \lim_{n\to\infty}\sup_{m\ge n}\frac{N^{-1}(2^{-m+n})}{N^{-1}(2^{-m})}.$$ 

As was said above, if $l_N$ is separable, we have $\alpha_{l_N}>0$. Then, comparing the latter formula for $k_-(U_N)$ with the first formula  in \eqref{equa24}, we get $k_-(U_N)<1$. Finally, combining this together with equivalence \eqref{equa26}, by Proposition \ref{prop:latttice and symm}, we conclude that $E_{l_N}=U_N$, completing the proof of our claim.

\end{document}